\documentclass[a4paper,11pt,reqno]{article}
\usepackage{a4wide}
\usepackage[english]{babel}
\usepackage{amssymb}
\usepackage{amsmath}
\usepackage{amsfonts}
\usepackage{amsthm}
\usepackage[pdftex]{graphicx}
\usepackage{hyperref}
\hypersetup{colorlinks,citecolor=red,filecolor=purple,linkcolor=blue,urlcolor=black}
\usepackage[T1]{fontenc}
\usepackage{enumerate}
\usepackage{relsize}

\newcommand{\equinf}[1]{\ \underset{#1\to+\infty}{\mathlarger{\mathlarger{{\sim}}}}\ }
\newcommand{\Nat}{\mathbb N}
\newcommand{\R}{\mathbb R}
\newcommand{\Esp}{\mathbb E}
\newcommand{\p}{\mathbb P}
\newcommand{\Einf}[1]{\Esp_\infty\left(#1\right)}
\newcommand{\Vinf}[1]{\Var_\infty\left(#1\right)}
\newcommand{\pinf}[1]{\p_\infty\!\left(#1\right)}
\newcommand{\linf}[1]{\underset{#1\to+\infty}\longrightarrow}
\newcommand{\1}[1]{\mathbf{1}\!_{\left\{#1\right\}}}
\newcommand{\EE}[1]{\Esp\left(#1\right)}
\newcommand{\Ek}[1]{\Esp_k\left(#1\right)}
\newcommand{\as}{\quad\mathrm{ a.s.}}
\newcommand{\eps}{\varepsilon}
\newcommand{\dif}{\mathrm{d}}
\newcommand{\Var}{\mathrm{Var}}
\newcommand{\me}{\medskip \noindent}
\newcommand{\bi}{\bigskip \noindent}

\def\be{\begin{eqnarray}}
\def\ee{\end{eqnarray}}
\def\ben{\begin{eqnarray*}}
\def\een{\end{eqnarray*}}

\newtheorem{prop}{Proposition}[section]
\newtheorem{defi}[prop]{Definition}

\newtheorem{lem}[prop]{Lemma}
\newtheorem{thm}[prop]{Theorem}
\newtheorem{rem}[prop]{Remark}
\newtheorem{cor}[prop]{Corollary}

\newcounter{example}
\title{Speed of coming down from infinity for birth and death processes}

\author{Vincent Bansaye\thanks{CMAP, Ecole Polytechnique, CNRS, route de
   Saclay, 91128 Palaiseau Cedex-France; E-mail: \href{mailto:vincent.bansaye@polytechnique.edu}{\texttt{vincent.bansaye@polytechnique.edu}}}, Sylvie M\'el\'eard\thanks{CMAP, Ecole Polytechnique, CNRS, route de
   Saclay, 91128 Palaiseau Cedex-France; E-mail: \href{mailto:sylvie.meleard@polytechnique.edu}{\texttt{sylvie.meleard@polytechnique.edu}}},
    Mathieu Richard\thanks{CMAP, Ecole Polytechnique, CNRS, route de
   Saclay, 91128 Palaiseau Cedex-France; E-mail: \href{mailto:mathieu.richard@cmap.polytechnique.fr}{\texttt{mathieu.richard@cmap.polytechnique.fr}}}}

\begin{document}
\maketitle

\begin{abstract}
We finely describe the speed of "coming down from infinity" for birth and death processes which eventually become extinct. 
Under general assumptions on the birth and death rates, we firstly determine the behavior of the successive hitting times of large integers.  
We put in light two different regimes depending on whether the mean time for the process  to go from  $n+1$ to $n$  is negligible or not compared to the mean time to reach $n$ from infinity. In the first regime, the coming down from infinity is very fast and the convergence is  weak.
In the second regime,  the coming down from infinity is gradual and
 a law of large numbers and a central limit theorem for the hitting times sequence hold. 
By an inversion procedure, we deduce   that the process is a.s. equivalent to  a non-increasing function when the time goes to zero.  Our results are illustrated by several examples including   applications to population dynamics and population genetics. The particular case where the death rate varies regularly is studied in details. 
\end{abstract}




\section{Introduction and main results}

$\quad$
Our goal in this paper is to finely describe the "coming down from infinity" for a birth and death process. We are motivated by the study of population dynamics and population genetics models 
 with  initially large populations.  For this purpose,
 we first  decompose  the trajectory of the process with respect to the hitting times of large integers. We then study the  small time behavior of the continuous time process when  it comes down from infinity. 

\medskip
The population size is modeled by  a birth and death process   $(X(t),t\geq0)$  whose birth rate (resp. death rate) at state $n\in \mathbb{N}$ is $\lambda_n$ (resp. $\mu_n$).
In the whole paper, the  rates $\lambda_n$ are nonnegative  and the rates $\mu_n$ are positive  for $n\geq1$. Moreover, we assume that     $\mu_0=\lambda_0=0$ for practical purpose. The latter implies that $0$ is an absorbing state.
Such processes have been extensively studied from the pioneering works on extinction  \cite{KarlinMcGregor57} and quasi-stationary distribution  \cite{Doorn1991}.

\medskip
It is well known \cite{KarlinMcGregor57, Karlin1975}  that
\begin{equation}\label{condition_extinction} 
\sum_{i\geq1}\frac{1}{\lambda_i\pi_i}=\infty
\end{equation}
is a necessary and sufficient
condition for  almost sure absorption of the process at $0$, where
\begin{equation*}
\label{defi_pi}
\pi_{1}= \frac{1}{\mu_1} \quad \textrm{ and for } \quad n\geq 2, \  \pi_n=\frac{\lambda_1\cdots\lambda_{n-1}}{\mu_1\cdots\mu_n}.
\end{equation*}

\noindent Under Condition \eqref{condition_extinction}, we first  define  the law $\p_{\infty}$ of the process starting from infinity with values in $\mathbb{N} \cup \{\infty\}$ (see Lemma \ref{defi_P_infini}) as the limit of the laws $\p_{n}$ of the process issued from $n$. 

\medskip
When the limiting process is non-degenerate, it hits finite values in finite time with positive probability. This behavior is captured
by the notion of "coming down from infinity". A key role is played by the decreasing sequence $(T_{n})_{n\geq0}$ of hitting times defined as 
$$
T_{n}:= \inf\{t\geq 0, X(t)=n\}.
$$
\noindent As proved in \cite[p.384]{TaylorKarlin1998} and in \cite[Chap.3]{Allen2011}, 
\begin{equation}
\label{temps}
\Esp_1(T_0)=\sum_{i\geq1}\pi_i\quad \textrm{ and }\quad  \Esp_{n+1}(T_{n})= \frac{1}{\lambda_n\pi_n}\sum_{i\geq n+1}\pi_i= \sum_{i\geq n+1}\frac{\lambda_{n+1}\cdots\lambda_{i-1}}{\mu_{n+1}\cdots\mu_i},
 \ \textrm{ for } n\geq 1.
 \end{equation}
 Remark that in case of pure-death process,  the law of $T_{n}$ under $\mathbb{P}_{n+1}$ is exponential with parameter $\mu_{n+1}$  and   for $n\in \mathbb{N}$,
 $\,\Esp_{n+1}(T_{n})= \frac{1}{\mu_{n+1}}$.
 
 \me
 Characterizations of the coming down from infinity have been given in \cite{Anderson1991, Cattiaux2009}.  They rely on  the convergence of the mean time of absorption when the initial condition goes to infinity or equivalently to the convergence of the series
\begin{equation}
\label{cv_{S}}
S=\lim_{n\to \infty} \Esp_{n}(T_{0})=\sum_{i\geq1}\pi_i+\sum_{n\geq1}\frac{1}{\lambda_n\pi_n}\sum_{i\geq n+1}\pi_i =\sum_{n\geq0} \left(\sum_{i\geq n+1}\frac{\lambda_{n+1}\cdots\lambda_{i-1}}{\mu_{n+1}\cdots\mu_i} \right)< +\infty.
\end{equation}
 This is  equivalent to the existence and uniqueness of the quasi-stationary distribution  at $0$ (see \cite{Doorn1991}, \cite{Cattiaux2009}) and to the finiteness of some exponential moments of $T_{0}$. Furthermore, monotonicity properties allow us to show that  this is also equivalent to instantaneous almost-sure coming down from infinity (Proposition \ref{CDI}). 
 
 \medskip
 In the whole paper, we suppose that Assumption \eqref{condition_extinction} holds and from Section \ref{section_comportement_T_n} onward, we assume that \eqref{cv_{S}} is satisfied, that is, the process instantaneously comes down from infinity. It guarantees the finiteness of all moments of $T_{n}$ under $\p_{n+1}$ and under $\p_{\infty}$, for which we have an explicit expression (Proposition \ref{expect-taun}).  In Section \ref{compTn}, we put in light two different regimes for the asymptotic behavior of ${T_{n}/ \Esp_{\infty}(T_{n})}$, depending on whether the mean time to go from  $n+1$ to $n$ 
 is negligible or not compared to the mean time to reach $n$ from $\infty$.   In the first regime, the coming down from infinity is very fast and the limit is random. In the second one, the coming down  is gradual  and due to the accumulation of small independent contributions, which leads to a law of large numbers. More precisely, we assume that
 $$\frac{\Esp_{n+1}(T_n)}{\Esp_{\infty}(T_n)}= \frac{\frac{1}{\lambda_n\pi_n}\sum_{i\geq n+1}\pi_i}{ \sum_{j\geq n}\frac{1}{\lambda_j\pi_j}\sum_{i\geq j+1}\pi_i } \linf{n} \alpha.$$
 
\medskip \noindent In the first (fast) regime ($\alpha>0$) and under the additional assumption, which is stronger than \eqref{condition_extinction}: 
\begin{equation}
\label{condition_limite}
l:=\lim_{n\to+\infty}\frac{\lambda_n}{\mu_{n}}<1,
\end{equation}
we prove that $T_n/\Esp_{\infty}(T_n)$ converges in law to a non-degenerate random variable  whose distribution is characterized by $l$ and $\alpha$.

\medskip\noindent In the second (gradual) regime ($\alpha=0$), we  prove a weak law of large numbers  under the following second moment assumption
 \begin{equation}\label{condition_croissance}
\sup_{n\geq0} \frac{\Esp_{n+1}(T_{n}^2)}{(\Esp_{n+1}(T_{n}))^2}
 <+\infty.\end{equation}
 More precisely, we prove that the sequence $(T_n/\Esp_{\infty}(T_n))$ converges in probability to $1$. Under some additional variance assumptions, we also obtain  a central limit theorem.
 
\me  Thanks to \eqref{temps} 
and to  forthcoming  \eqref{mom2}, we note that both expectations in  \eqref{condition_croissance} can be written in terms of the birth and death rates.
Condition
  \eqref{condition_croissance} is fulfilled in many cases we have in mind. For instance,
it  holds for pure death processes.

\me  We will see in the next section some more tractable conditions ensuring \eqref{condition_extinction}, \eqref{cv_{S}},
 \eqref{condition_limite} and \eqref{condition_croissance}. 
 

\me In the second regime, under \eqref{condition_croissance}  and the following  additional  condition
\be
\label{hyp-carre} \displaystyle\sum_{n\geq0}\left(\frac{\Esp_{n+1}(T_{n})}{\Esp_\infty(T_n)}\right)^2<+\infty,\ee
which means that the convergence of $\Esp_{n+1}(T_n)/\Esp_{\infty}(T_n)$ to $0$ is fast enough, one also get a strong law of large numbers for $T_{n}/ \Einf{T_{n}}$.

\bi
 We then derive in Section \ref{section_comportement_X} the small time behavior of the process $X$. We prove that 
 \begin{equation*}
 \lim_{t\to0}\frac{X(t)}{v(t)}=1,
\end{equation*}
where $v$ is   the generalized inverse function of
$n\mapsto \Esp_{\infty}(T_n)= \sum_{j\geq n}\frac{1}{\lambda_j\pi_j}\sum_{i\geq j+1}\pi_i \,$:
\begin{equation*}
v(t)=\inf\{n\geq0;\ \Esp_{\infty}(T_n)\leq t\}.
\end{equation*}
The limit holds in probability in the first regime. Remark that in this fast case, $T_{n}/\Esp_{\infty}(T_n)$ converges in law (an not in probability) to a random variable but nevertheless, $X(t)$ behaves as $v(t)$ for $t$ small. That is due to the fact that $\Esp_{\infty}(T_{[nx]})$ is negligible with respect to $\Esp_{\infty}(T_n)$ for any $x>1$ and for large $n$. In the second regime, one needs some additional assumptions and  almost sure convergence can be obtained. The proof relies on two ingredients: the short time behavior of the non-increasing process equal to $n$ on $[T_{n}, T_{n-1}[$ and  the 
control of the height of the excursion of the  process $X$ during the time interval $[T_n,T_{n-1})$. Technical assumptions are required in 
the second regime to estimate the variations of $\Esp_{\infty}(T_n)$ and  to deduce the behavior of $X$ from that of $(T_{n})_{n}$, by a non trivial inversion procedure. Our  motivations
and applications from population dynamics and population genetics  meet these assumptions. Thus, our results cover  general birth and death models including many different ecological scenarios, as  competition models with polynomial death rates  \cite{Sibly22072005}  or Allee effect \cite{Kot2001}. Lambert \cite{Lambert2005} characterizes the distribution of the  absorption
 time for the logistic branching process  starting from infinity. Our work extends in different way    the case of  Kingman coalescent, for which speed of coming down from infinity has already been obtained by Aldous \cite{Aldous99}. More generally,   in the gradual regime, the behavior of the process coming down from in finity is similar to that of $\Lambda$ coalescent obtained in \cite{Berestiycki_Limic, LT}.  
 We also note that our approach relies on the decomposition of the trajectory of $X$ with respect to the reaching times of the successive integers and our results could be  extended to processes with several births. Another motivation for the results below is the study of birth and death process in random environment and in particular the study of the regulation of the population during
unfavorable periods. This latter is a work in progress. 

\me
The paper is organized as follows. In the next section, we work under the absorption assumption \eqref{condition_extinction} and prove the existence of the law of the process starting from infinity. Thus we gather general characterizations of the coming down from infinity and we show the equivalence to the a.s. instantaneous  coming down.
Focusing in Section \ref{section_comportement_T_n} on birth and death processes satisfying \eqref{condition_limite} or \eqref{condition_croissance}, 
we  describe the hitting times  of large integers.
 In Section \ref{section_comportement_X}, we  obtain a law of large numbers describing the small
time behavior of the process $X$. Examples and applications are provided all along the paper and illustrate the different regimes. The last Section \ref{ApplVR} focuses on regularly varing death rate and provides our main application to small time behavior for population dynamics and population genetics one-dimensional processes coming down from infinity.

\section{Preliminaries and  coming down from infinity}
\label{section_CDI}

\subsection{Preliminaries}
The first lemma allows us to define  the law of the process starting from infinity. It is  based on monotonicity arguments following Donnelly \cite{Donnelly91}. We set $\overline{\Nat}:=\{0,1,\ldots \} \cup\{\infty\}$ and for any $T>0$, we denote by $\mathbb{D}_{\overline{\Nat}}([0,T])$ the Skorohod space of c\`adl\`ag functions on $[0,T]$ with values in $\overline{\Nat}$.

\begin{lem}
\label{defi_P_infini}
Under \eqref{condition_extinction},  the sequence $(\p_{n})_{n}$ converges weakly in the space of  probability  measures on  $\mathbb{D}_{\overline{\Nat}}([0,T])$ to a probability measure $\p_{\infty}$.
\end{lem}

\noindent At this point, the limiting process is not assumed to be finite for positive times. 

\begin{proof}
We follow  the tightness argument given in the first part of  the proof of Theorem 1 by Donnelly in \cite{Donnelly91}. Indeed,
 no integer is an instantaneous state for the process ($\lambda_{n}, \mu_{n} <\infty$ for each $n\geq 0$) and the process
 is stochastically monotone with respect to the initial condition. It ensures that Assumption (A1)
of \cite{Donnelly91} holds.  In addition, Assumption \eqref{condition_extinction} ensures that the process almost surely does not explode and  (A2) of  \cite[Thm. 1]{Donnelly91} is also satisfied by  denoting $B_n^N$ the birth and death process
$X$ issued from $n$ and stopped in $N$.   

\noindent Then the tightness holds and  we identify the finite marginal distributions by noticing that for $k\geq1$, for $t_1,\dots,t_k\geq0$ and for $a_1,\dots, a_k\in\Nat$, the quantities $\p_{n}(X(t_{1})\leq a_{1}, \cdots, X(t_{k})\leq a_{k})$ are non-increasing with respect to $n\in\Nat$ (and thus converge). \end{proof}

\me We now focus on the time spent by the process $(X(t), t\geq 0)$  to go from level $n+1$ to level $n$. For $n\geq 0$, we introduce the function 
$$ G_n(a):=\Esp_{n+1}(\exp(-aT_n)), \  a>0.$$

\me
\begin{prop}
\label{expect-taun} Suppose that \eqref{condition_extinction} holds.
 For any $a>0$ and $n\geq1$, we have
\begin{equation}\label{recurrence_Gn}
G_{n}(a)=1+\frac{\mu_n+a}{\lambda_n}-\frac{\mu_n}{\lambda_n}\frac{1}{G_{n-1}(a)}.
\end{equation}
Moreover, for every $n\geq 0$,
\be
\label{mom2}
\Esp_{n+1}(T_{n}^2)=\frac{2}{\lambda_n\pi_n}\sum_{i\geq n} \lambda_i\pi_i\Esp_{i+1}(T_{i})^2, \quad 
\Esp_{n+1}(T_{n}^3)=\frac{6}{\lambda_n\pi_n}\sum_{i\geq n} \lambda_i\pi_i\Esp_{i+1}(T_{i})\Var_{i+1}(T_{i}).
\ee
\end{prop}
\medskip \begin{proof}  
\noindent We denote by $\tau_n$ a random variable distributed as $T_n$ under $\p_{n+1}$ and consider the Laplace transform of $\tau_n$. Following  \cite[p. 264]{Anderson1991} and
by the Markov property, we have
$$\tau_{n-1}\overset{(\mathrm{d})}=\1{Y_n=-1}E_n+\1{Y_n=1}\left(E_n+\tau_{n}+\tau'_{n-1}\right)$$
where $Y_n$, $E_n$, $\tau'_{n-1}$ and $\tau_{n}$ are independent random variables,
 $E_n$ is an exponential  random variable with parameter $\lambda_n+\mu_n$ and $\tau'_{n-1}$ is distributed as $\tau_{n-1}$ and $\p(Y_n=1)=1-\p(Y_n=-1)=\lambda_n/(\lambda_n+\mu_n)$.
Hence, we get
$$G_{n-1}(a)=\frac{\lambda_n+\mu_n}{a+\lambda_n+\mu_n}\left(G_n(a)G_{n-1}(a)\frac{\lambda_n}{\lambda_n+\mu_n}+\frac{\mu_n}{\lambda_n+\mu_n}\right)$$
and \eqref{recurrence_Gn} follows. 

\medskip \noindent Differentiating \eqref{recurrence_Gn} twice at $a=0$, we get
\begin{equation*}\label{eq_1750}
\Esp_{n}(T_{n-1}^2)=\frac{\lambda_n}{\mu_n}\Esp_{n+1}(T_{n}^2)+2\Esp_{n}(T_{n-1})^2, \quad n\geq1.
 \end{equation*}
In the particular case when $\lambda_{N}=0$ for some $N>n$, 
a simple induction gives
\be
\label{caspart}
\Esp_{n+1}(T_{n}^2)=\frac{2}{\lambda_n\pi_n}\sum_{n\leq i \leq N-1} \lambda_i\pi_i\Esp_{i+1}(T_{i})^2 
\ee
and  \eqref{mom2} is proved. 
In the general case, let $N > n$. Thanks to Assumption \eqref{condition_extinction}, $T_0$ is  finite and the process a.s. does not explode in finite time for any initial condition. Then $T_n$ is  finite and
$T_N \rightarrow +\infty$ $\p_{n+1}$-a.s., where we use
 the convention $T_N=+\infty$ on the event $\{\forall t\geq 0 : X(t) \ne N\}$. The monotone convergence theorem yields 
$$\Esp_{n+1}(T_n^2 ;  T_n\leq T_N )\linf{N}
\Esp_{n+1}(T_n^2).$$ Let us  consider a birth and death process $X^N$ with birth and death rates $(\lambda^N_k,\mu^N_k : k\geq 0)$ such that $(\lambda^N_k,\mu^N_k)= ({\lambda}_k,{\mu}_k)$ for $k\ne N$ and ${\lambda}^N_N=0, {\mu}^N_N=\mu_N$. 

\noindent Since
$(X_t :  t\leq T_N)$ and $({X}^N_t : t\leq {T}^N_N)$ have the same distribution under $\p_{n+1}$, we get 
$$\Esp_{n+1}\left(T_n^2 ;  T_n\leq T_N \right)=\Esp_{n+1}\left((T_n^N)^2 ;  T_n^N\leq T_N^N \right),$$
which yields $$\Esp_{n+1}(T_n^2)=\lim_{N\rightarrow \infty} \Esp_{n+1}\left((T_n^N)^2 ;  T_n^N\leq T_N^N \right)\leq \lim_{N\rightarrow \infty} \Esp_{n+1}\left((T_n^N)^2 \right),$$
where the convergence of the last term is due to the stochastic monotonicity of $T_n^N$  with respect to $N$ under $\p_{n+1}$.
Since  $T_n^N$ is stochastically smaller than $T_n$ under $\p_{n+1}$, we have also
$$\Esp_{n+1}((T_n)^2)\geq \Esp_{n+1}((T_n^N)^2).$$
We deduce that
$$\Esp_{n+1}((T_n)^2)=\lim_{N\rightarrow \infty}\Esp_{n+1}((T_n^N)^2)=\lim_{N\rightarrow \infty}\frac{2}{\lambda_n\pi_n}\sum_{n\leq i \leq N-1} \lambda_i\pi_i\Esp_{i+1}(T_{i}^N)^2, $$
where the last identity comes from \eqref{caspart}.
Adding that $\Esp_{i+1}(T_{i}^N)$ is non-decreasing with respect to $N$ yields the expected expression for $\Esp_{n+1}((T_n)^2)$ by monotone convergence.

\medskip \noindent  The third moment is obtained similarly
 by differentiating \eqref{recurrence_Gn} three times, which gives the recurrence equation
 \begin{equation*}
\Esp_{n}(T_{n-1}^3)=\frac{\lambda_n}{\mu_n}\Esp_{n+1}(T_{n}^3)+6\Esp_{n}(T_{n-1})\Var_{n}(T_{n-1}), \quad n\geq1.
 \end{equation*}
The coupling argument we have used above allows us to conclude.
\end{proof}

\begin{rem} Using Proposition \ref{expect-taun}, \eqref{condition_croissance} writes
\be
\label{condition_croissancebis}
\sup_{n\geq 0} \sum_{i\geq n} \frac{\lambda_{i}\pi_{i}}{\lambda_{n}\pi_{n}}\left(\frac{\Esp_{i+1}(T_{i})}{\Esp_{n+1}(T_{n})}\right)^2<+\infty.
\ee
\end{rem}

\subsection{Instantaneous coming down from infinity}
\noindent  We now define a strong notion of coming down from infinity corresponding to the behavior of birth and death  processes under \eqref{condition_extinction} and \eqref{cv_{S}}:   the process comes down instantaneously almost surely.

\begin{defi}\label{defi_CDI} 
 The process $(X(t), t\geq 0)$ instantaneously comes down from infinity if for any $t>0$,
 \be
 \label{ICD}\lim_{m\rightarrow  \infty}\lim_{k\to+\infty}\p_k(T_m<t)=1.\ee
\end{defi}

\noindent Note that  \eqref{ICD} is equivalent to 
$$\p_{\infty}(\forall t>0, X(t)<+ \infty)=1.$$

\noindent Let us now show that \eqref{ICD} is satisfied under \eqref{condition_extinction} and \eqref{cv_{S}}. In fact we  give  several necessary and sufficient conditions for $(X(t), t\geq 0)$ to come down from  infinity. The first two ones are directly taken  from \cite{Cattiaux2009}. We add here an exponential moment criterion. 
We  also mention that  it is equivalent to the existence (cf.  \cite{Doorn1991}) and uniqueness (cf. \cite{Cattiaux2009}) of a quasi-stationary distribution for the process $X$.
\begin{prop}\label{CDI}
Under condition \eqref{condition_extinction}, the following assertions  are equivalent:
\begin{enumerate}[{\normalfont (i)}]
\item The process $(X(t), t\geq 0)$  instantaneously comes down from infinity.
\item  Assumption \eqref{cv_{S}} is satisfied: $S<+\infty$.
\item $\sup_{k\geq0}\Esp_k [T_0]<+\infty$.
\item For all $a>0$, there exists $k_a\in\Nat$ such that $\sup_{k\geq k_a}\Esp_k\left(\exp(aT_{k_a})\right)<+\infty$.
\end{enumerate}
\end{prop}
\me

 \me
Proposition \ref{CDI}  implies in particular that  under \eqref{condition_extinction} and \eqref{cv_{S}}, the moments of $T_n$ under $\p_n$ and $\p_{\infty}$ are finite. Moreover  their explicit expression can be derived from Proposition \ref{expect-taun} and will be useful in the rest of the paper.

 \medskip \noindent 
\begin{rem}
It is proved  in the next Lemma \ref{lemme_moments_tau_n} (i), that the equivalence between Assertion (ii), and then (i), (iii), (iv), is equivalent to the convergence of the series $\sum_{i\geq1} 1/\mu_i$, as soon as forthcoming Assumption \eqref{condition_croissance1} is satisfied. 
\end{rem}

\me
\begin{proof}[Proof of Proposition \ref{CDI}]
Assertion (i) implies (ii), and (ii) and (iii) are equivalent according to \cite[Prop 7.10]{Cattiaux2009}.  We now prove that
(iv) is equivalent to (ii) and that (ii) implies (i).

\smallskip \noindent First, we check that (iv) implies that $X$ comes down from infinity, which means that $+\infty$ is an entrance boundary. Then it well known that (ii) holds (see Section 8.1 in  \cite{Anderson1991} or Proposition 7.10 in \cite{Cattiaux2009}). Indeed, taking $a=1$ in (iv), we have $M:=\sup_{k\geq k_1}\Esp_k\left(\exp(T_{k_1})\right)<+\infty$. Then, Markov inequality ensures that  for all  $k\geq k_1$ and $t\geq0$,
$\p_k(T_{k_1}<t)\geq 1-\exp(-t)M$. Choosing  $t$ large enough ensures that the process comes down from infinity.

\noindent We then prove that (ii) implies (iv)  by adapting the proof of \cite[Prop 7.6]{Cattiaux2009}  to the discrete setting.
We fix $a>0$ and using  $S<+\infty$, there exists $k_a>1$ such that
\begin{equation*}\label{defi_n_a}\sum_{n\geq k_a-1}\frac{1}{\lambda_n\pi_n}\sum_{i\geq n+1}\pi_i\leq\frac{1}{a}.\end{equation*}
We now define the Lyapounov function  $J_a$ as
$$J_a(m):=\left\{\begin{array}{cc}
\displaystyle\sum_{n=k_a-1}^{m-1}\frac{1}{\lambda_n\pi_n}\sum_{i\geq n+1}\pi_i& \textrm{if } m\geq k_a\,,\\
0&\textrm{ if }m<k_a\,.
\end{array}\right.$$
We notice that $J_{a}$ is  non-decreasing and bounded and we introduce  the infinitesimal generator $L$ of $X$,  defined by 
$$L(f)(n)=\left(f(n+1)-f(n)\right)\lambda_n+\left(f(n-1)-f(n)\right)\mu_n,$$
for any bounded function $f$  and any $n \geq 1$. Then, the process
 $$M_t:=e^{at}J_a(X(t))-\int_0^te^{au}\left(aJ_a(X(u))+LJ_a(X(u))\right)\dif u, \qquad (t\geq0)$$
is a martingale with respect to the natural filtration of $X$.
Adding that $LJ_a(m)=-1$ for any $m\geq k_a$ and that  $J_a(X(u))\leq J_a(\infty)\leq 1/a$ , we get  for all $k\geq k_a$ and $t\geq0$,
\begin{eqnarray*}
\Ek{e^{at\wedge T_{k_a}}J_a(X(t\wedge T_{k_a}))}&=& \Ek{\int_0^{t\wedge T_{k_a}}e^{au}\left(aJ_a(X(u))+LJ_a(X(u))\right)\dif u}+J_a(k)\\
&= &\Ek{\int_0^{t\wedge T_{k_a}}e^{au}\left(aJ_a(X(u))-1\right)\dif u}+J_a(k) \\
&\leq & J_a(k).
\end{eqnarray*}
Adding that  for any  $k\geq k_a$, $\p_k$-a.s. $J_a(X(t\wedge T_{k_a}))\geq J_a(k_a)$, we get  $\ \Ek{e^{at\wedge T_{k_a}}}\leq \frac{J_a(k)}{J_a(k_a)}$.
Then (iv) follows from the monotone convergence theorem and Assumption (ii).

\me 
It remains to show that (ii)  implies (i). 
On the one hand, according to \eqref{temps}, $\Esp_\infty(T_{n})=\sum_{i\geq n} \Esp_{i+1}(T_i)$ and Assumption (ii) entails that $\Esp_\infty(T_{n})$ vanishes as $n\to\infty$ as the rest of the finite series $S$.
On the other hand, under $\p_{\infty}$, the sequence $(T_{n})_{n\geq0}$ decreases to some random variable $\,T_{[0,\infty)}$. Then, from the monotone convergence theorem, $\Esp_\infty(T_{n})$ decreases to $\Esp_\infty(T_{[0,\infty)})$ and  $\Esp_\infty(T_{[0,\infty)})=0$. It ensures that $T_{[0,\infty)}=0$ $\p_{\infty}$ a.s. and $X$ instantaneously comes down from infinity. The proof is then complete.
\end{proof}
\subsection{More tractable conditions}
\label{tract}
Let us give some tractable conditions ensuring \eqref{condition_extinction}, \eqref{cv_{S}},
 \eqref{condition_limite} or  \eqref{condition_croissance}, which will be useful for examples and applications.

\begin{lem}
\begin{enumerate}[{\normalfont(i)}]
\item 
Under Assumption 
\begin{equation}\label{condition_croissance1}
\sup_{n,i\geq1}\frac{\mu_n}{\mu_{n+i}}<+\infty, \qquad \limsup_{n\rightarrow \infty} \frac{\lambda_n}{\mu_{n}}<1,
\end{equation}
\noindent Condition  \eqref{cv_{S}}  holds  if and only if 
\be\label{crit-CDI}\ \displaystyle\sum_{n\geq1}\frac{1}{\mu_n} <+\infty.\ee

\item Assuming that \be
\label{condtech}
 \sup_{n,i\geq 1} \frac{\mu_n}{\mu_{n+i}}<+\infty, \qquad \frac{\lambda_n}{\mu_n} \underset{n\rightarrow \infty}{\longrightarrow}  0, \qquad \sum_{n \geq 1} \frac{1}{\mu_n}<+\infty,
 \ee
 then  \eqref{condition_extinction}, \eqref{cv_{S}},
 \eqref{condition_limite} and \eqref{condition_croissance} are satisfied and 
$$\Esp_{n+1}(T_n^k)\equinf{n}\frac{k!}{\mu^k_{n+1}},\quad \hbox{ for }\ k=1,2,3.$$
\end{enumerate}\label{lemme_moments_tau_n}
\end{lem}

\me Criterion \eqref{crit-CDI} can be seen as the discrete counterpart of the criterion in \cite[p.1953]{Cattiaux2009} stating that the Feller diffusion process $Z$  defined by  $\dif Z_t=\sqrt{\gamma Z_t}dB_t+ Z_{t}(r- f(Z_t))\dif t$  (for a suitable    function $f$ and $r>0$), comes down from infinity if and only if $\int_1^\infty\frac{\dif x}{xf(x)}<+\infty$.

\me
\begin{proof}
%
  We begin with the proof of point (i).  Coming back to  \eqref{temps}, 
the first term of the series giving $\Esp_{n+1}(T_{n})$ is $1/\mu_{n+1}$, hence  $\Esp_{n+1}(T_{n})\geq 1/\mu_{n+1}$.  Moreover, using the second part of Assumption \eqref{condition_croissance1}, there is $l'<1$ such that for $n$ large enough, 
$\lambda_n / \mu_n\in [0,l')$ and then for $n$ large enough
$$\frac{1}{\mu_{n+1}}\leq \Esp_{n+1}(T_{n})\leq\sum_{j\geq1}l'^{j-1}\frac{1}{\mu_{n+j}}\leq\frac{1}{1-l'}\frac{1}{\mu_{n+1}}\sup_{n\geq 1,k\geq 0} \frac{\mu_n}{\mu_{n+k}}.$$
Then the first part of Assumption \eqref{condition_croissance1} allows to get  (i).

 \medskip\noindent
 Under the assumptions \eqref{condtech}, the properties \eqref{condition_extinction}, \eqref{condition_limite} and \eqref{condition_croissance} are obvious, whereas \eqref{cv_{S}} is a consequence of point (i) of the lemma.
 
\me  To get the asymptotic behavior of the moments of $T_n$, we use the expression of $\Esp_{n+1}(T_n^k)$ provided in Proposition \ref{expect-taun} and the fact that   $\lambda_n/\mu_n$ goes to $0$. Then, for $k\in\{1,2,3\}$, by induction we can write
$$\Esp_{n+1}(T_n^k)=k!b_{n,k}(1+A_{n,k}),$$
where $b_{n,1}=1/\mu_{n+1}$, $b_{n,2}=(\Esp_{n+1}(T_n))^2$ and $b_{n,3}=\Esp_{n+1}(T_n)\Var_{n+1}(T_n)$
and $A_{n,k}\rightarrow 0$ as $n\rightarrow \infty$, which will complete the proof.
\noindent 
Indeed, for $k=1$, we know from \eqref{temps} that
 $\, \Esp_{n+1}(T_{n}) =\frac{1}{\mu_{n+1}} \left[1+\sum_{i\geq n+2}\frac{\lambda_{n+1}\cdots\lambda_{i-1}}{\mu_{n+2}\cdots\mu_i}\right].$
Moreover for every $l' \in(0,1)$ and  $n$ large enough, we have  $\lambda_n/\mu_n <l'$ and
 $$\frac{\lambda_{n+1}\cdots\lambda_{i-1}}{\mu_{n+2}\cdots\mu_i} \leq  l'^{i-(n+1)}.\sup_{n \geq 1,k\geq 0} \frac{\mu_n}{\mu_{n+k}},$$
 which ensures that $\Esp_{n+1}(T_n)\sim 1/\mu_{n+1}$. Combining this equivalence and  the expression of $\Esp_{n+1}(T_n^2)$ provided in Proposition \ref{expect-taun} yields similarly the asymptotic behavior of the second moment ($k=2$) and then the third moment $(k=3)$.
 \end{proof}

\begin{rem}Our original motivations for considering the coming down from infinity of birth and death processes
 are
  the regulation of large populations due to competition and the short time behavior of branching coalescing   models  (see e.g.  \cite{Ancgraph} for some  motivations  for ancestral graphs).
In this context, the birth rate is usually  linear, which corresponds to independent reproduction events,  or even zero for pure coalescing models. The death rate is often quadratic such as 
for Kingman coalescent and logistic competition, but polynomial death rate may be relevant, see  in particular  \cite{Sibly22072005} for a statistical study of the  death rate due to competition. Thus, we are interested in the particular case  $\lambda_n\leq Cn$ for some $C>0$ and
  $\mu_n=n^\rho\log^\gamma n$ with $\rho>1$. In this case, Assumption \eqref{condtech} is obviously satisfied and then \eqref{condition_extinction}, \eqref{cv_{S}},
 \eqref{condition_limite} and \eqref{condition_croissance} hold. Proposition \ref{CDI} ensures that the process comes down a.s. instantaneously from infinity and has bounded exponential moments.
We refer to Section \ref{ApplVR} for the fine description of this coming down from infinity. 

 \end{rem}

\section{Asymptotic behavior of \texorpdfstring{$T_n$}{Tn} \texorpdfstring{under $\p_\infty$}{}}\label{section_comportement_T_n}
\label{compTn}
From now on, we consider  sequences $(\lambda_n)_{n\geq0}$ and $(\mu_n)_{n\geq0}$  satisfying the hypotheses \eqref{condition_extinction}
and \eqref{cv_{S}}.  Thus,  according to Lemma \ref{defi_P_infini} and Proposition \ref{CDI},  $\p_\infty$ is well-defined and $X$ strongly comes down from infinity. Moreover $T_n<+\infty$ $\p_{\infty}$ a.s.
for any $n\geq 0$. In this section, we study the asymptotic
behavior of $T_n$ as $n\to+\infty$ under $\p_\infty$.
Let us recall that $\Esp_{\infty}(T_n)=\sum_{i\geq n} \Esp_{i+1}(T_i)$, so that
\eqref{temps} yields
$$\Esp_\infty(T_n)=\sum_{i\geq n}\frac{1}{\lambda_{i}\pi_{i}}\sum_{j\geq {i+1}}\pi_j.$$
Then, $S<+\infty$ ensures that $\Esp_\infty(T_n)$ decreases to $0$ as $n\to+\infty$.

\medskip \noindent
In the following two subsections, we compare $T_n$ to its mean $\Einf{T_n}$ as $n\to+\infty$. 
Two regimes appear depending on whether the ratio of  mean times $\Esp_{n+1}(T_{n})/\Esp_\infty (T_n)$ converges to a non-degenerate value or vanishes.
In the first case (fast regime - Theorem \ref{Behavior_Tnb}), the process comes down very quickly from infinity,  $T_n$ is then essentially  the time spent close to $n$ and renormalizing $T_n$ by its mean yields a random limit.
In the second case (gradual regime - Theorem \ref{Behavior_Tn}),  $T_n$ can be seen as the contribution of a large number of independent random variables and the limit  equals $1$. 

\noindent 
In both cases, the proofs rely on the fact that $T_n=\sum_{i\geq n} \tau_i$  $\p_{\infty}$-a.s., where for $n\geq 0$,  the  random variable $\tau_n$ is the time spent between $T_{n+1}$ and $T_n$ : 
$$\tau_n:=\inf\{t>T_{n+1};X(t)=n\}-T_{n+1}.$$
By the strong Markov property,  the  random variables $(\tau_i)_{i\geq 0}$ are  independent (under $\p_{\infty}$) and $\tau_i$ is distributed as $T_i$ under $\p_{i+1}$.
In the sequel of the section, we use for $n\geq 0$ the notation $$m_n:=\Esp(\tau_{n})=\Esp_{n+1}(T_n),\quad
r_n:=\frac{m_n}{\Esp_\infty(T_n)}=\frac{\Esp_{n+1}(T_{n})}{\Esp_\infty(T_n)}.$$
Examples which illustrate the two regimes
and the two convergences are provided in forthcoming Section \ref{exemples}, while an application to the regularly varying case  is  developed in  Section \ref{ApplVR}. 

\subsection{The fast regime}

\me
\begin{thm}\label{Behavior_Tnb}
We assume that \eqref{condition_extinction}, \eqref{cv_{S}} and \eqref{condition_limite} hold and
 $$\frac{\Esp_{n+1}(T_{n})}{\Esp_\infty(T_n)}\linf{n}\alpha,$$ with $\alpha\in(0,1]$.  Then,
$$ \qquad \qquad  \frac{T_n}{\Esp_\infty(T_n)}\overset{\mathrm{(d)}}{\underset{n\to+\infty}\longrightarrow}Z:=\sum_{k\geq0}\alpha\left(1-\alpha\right)^kZ_k, \qquad \qquad \qquad$$
where $(Z_k)_k$ is a sequence of i.i.d. random variables whose common Laplace transform $G(a):=\Esp_\infty\left(\exp(-aZ_0)\right)$ is the unique function $[0,+\infty)\to[0,1]$  that satisfies
\begin{equation}\label{equation_G}
\forall a>0, \quad  G(a)\left[l\big(1-G(a(1-\alpha))\big)+1+a(1-l(1-\alpha))\right]=1.
\end{equation}
\end{thm}

\me
\noindent We note that when $\alpha=1$, $Z=Z_0$ is an exponential random variable   with parameter $1$.

\bigskip \noindent 
\emph{Example}:  If $\mu_n=(n!)^\gamma$ with $\gamma>0$, $\Einf{T_n}\sim ((n+1)!)^{-\gamma}$. Hence, $\lim_{n\to+\infty} \Esp(\tau_n)/\Einf{T_n}=1$ and Theorem \ref{Behavior_Tnb} (i) yields
$$((n+1)!)^\gamma T_n\overset{\textrm{(d)}}{\linf{n}}E,$$ where $E$ is an exponential r.v. with parameter $1$. Another example is studied in forthcoming Section \ref{speedfast}. 

\vskip 1cm
\noindent  Before proving Theorem \ref{Behavior_Tnb}, let us show the following key lemma, which focuses on the asymptotic behavior of the distribution of   $(\tau_n)_{n}$.

\me
\begin{lem}\label{equivalents_alpha_non_nul}
 If $\lim_{n\to+\infty} r_n=\alpha\in(0,1]$, we have
$$\frac{\tau_n}{m_n}\overset{\mathrm{(d)}}{\linf{n}}\zeta,$$
where the Laplace transform of $\zeta$ is the unique solution of  \eqref{equation_G}.
\end{lem}

\begin{proof} Recalling $\lambda_n/\mu_n \rightarrow l$ as  $n\rightarrow \infty$, let us first check that
\begin{equation}\label{equivalents_ratios_alpha}
\lim_{n\to+\infty}\frac{\Esp_\infty(T_{n+1})}{\Esp_\infty(T_n)}=\lim_{n\to+\infty}\frac{m_{n+1}}{m_n}=1-\alpha, \qquad
\lim_{n\to+\infty}\mu_nm_{n-1}=\frac{1}{1-l(1-\alpha)}.
\end{equation}
The first part of  \eqref{equivalents_ratios_alpha} comes from  $\Esp_\infty(T_{n+1})/\Esp_\infty(T_n)=1-r_n$ and $$\frac{m_{n+1}}{m_n}=\frac{r_{n+1}}{r_{n}}\frac{\Esp_\infty(T_{n+1})}{\Esp_\infty(T_n)}.$$
Moreover,  differentiating \eqref{recurrence_Gn} at $a=0$ yields
\begin{equation*}\label{recurrence_m_n}
1=\frac{\lambda_n}{\mu_n}\frac{m_{n}}{m_{n-1}}+\frac{1}{\mu_n m_{n-1}}
\end{equation*}
and using \eqref{condition_limite}  gives the second part of \eqref{equivalents_ratios_alpha}. 

\medskip
\noindent Let us  prove the uniqueness of  the function satisfying \eqref{equation_G}. 
 For any bounded function $g:[0,+\infty)\to[0,1]$, we define the function $H(g):[0,+\infty)\to[0,1]$ as
$$H(g):a\longmapsto \left[1+a(1-l(1-\alpha))+l(1-g(a(1-\alpha))\right]^{-1}.$$
For two functions $g_1$ and $g_2$ and any $a>0$, we have
{\setlength\arraycolsep{2pt}
\begin{eqnarray*}
\left|H(g_1)(a)-H(g_2)(a)\right|  &=& H(g_1)(a)H(g_2)(a) l\left|g_1(a(1-\alpha))-g_2(a(1-\alpha))\right|\\
\end{eqnarray*}}
and using that for any $a>0$, $H(g_1)(a)\leq 1,$
\begin{equation}\label{contrac} 
\|H(g_1)-H(g_2)\|_\infty\leq l\|g_1-g_2\|_\infty,
\end{equation}
which ensures the expected uniqueness since $l<1$. 

\medskip
\noindent We  now prove the convergence  in distribution of $\tau_n/m_n$ as $n\to+\infty$. For $n\geq0$, let $F_n:[0,+\infty)\longrightarrow[0,1]$ be  defined as
$$F_n(a):=\Esp(\exp(-a\tau_n/m_n))=\Esp_{n+1}(\exp(-aT_n/m_n))=G_{n}\left(\frac{a}{m_{n}}\right),\quad (a>0).$$
%
By  \eqref{recurrence_Gn},   for all  $a>0$ and $n\geq1$, we have
\begin{equation*}
G_{n-1}\left(\frac{a}{m_{n-1}}\right)=\left[1+\frac{a}{\mu_n m_{n-1}}+ \frac{\lambda_n}{\mu_n}\left(1-G_{n}\left(\frac{a}{m_{n-1}}\right)\right) \right]^{-1},
 \end{equation*}
which we rewrite as
\begin{equation}\label{recurrence_Gn2}
F_{n-1}=H_n(F_n),
\end{equation}
where for every function $f : [0,\infty) \rightarrow [0,1]$, $n\geq 1$  and $a\geq 0$,
$$H_n(f)(a)=\left[1+\frac{a}{\mu_n m_{n-1}}+ \frac{\lambda_n}{\mu_n}\left(1-f\left(a\frac{m_n}{m_{n-1}}\right)\right) \right]^{-1}.$$
Using \eqref{equivalents_ratios_alpha},  we have for every $a \geq 0$,
$$ \sup_{  f  \in \mathcal C^1_{1}}  | H_n(f)(a) -H(f)(a) | \linf{n} 0,$$
with $\mathcal C^1_{1} :=\{ f  \in \mathcal{C}^1([0,\infty),  [0,1]) :  \| f'\|_{\infty} \leq 1\}.$
Moreover, $F_n=H_{n+1}\circ  \ldots \circ H_{n+k}(F_{n+k})$ and by triangle inequality
\begin{multline*}\left| F_n(a)-H^{\circ k}(F_{n+k})(a)\right|\leq\Big| H_{n+1}(F_{n+1})(a)-H(F_{n+1})(a)\Big|\\
+\left|H(F_{n+1})(a)-H(H^{\circ k-1}(F_{n+k}))(a)\right|.
\end{multline*}
Adding that for every $n$, $F_n \in \mathcal C^1_{1}$ and recalling \eqref{contrac}, we get by induction over $k\geq 0$ that
$$F_n(a)-H^{\circ k}(F_{n+k})(a)\linf{n}0$$
for all $a\geq 0$ and $k\geq 0$.
We use  again \eqref{contrac} to obtain that
$$\|H^{\circ k}(F_{n+k})-H^{\circ k}(\mathbf{1})\|_\infty\leq l^k\|F_{n+k}-1\|_\infty\leq l^k.$$
Recalling that $l <1$, we can combine the two last displays and 
 for each $\epsilon>0$, we can find $k$ such that for $n$ large enough
$$| F_n(a)-H^{\circ k}(\mathbf{1})(a)| \leq 2\epsilon.$$
Thus, $(F_n(a) : n\geq 0)$ is a Cauchy sequence and $F_n(a)$ converges to $F(a)$ on $[0,\infty)$. The fact that $F_n  \in \mathcal C^1_{1}$ ensures that this convergence is uniform in  each compact set. Letting $n\rightarrow \infty$ in \eqref{recurrence_Gn2} then yields $F=H(F)$, which means that $F$ satisfies \eqref{equation_G}. 

\noindent Finally, we check   that $F$ is the Laplace transform of some random variable by proving that $F(0^+)=\lim_{a\to0}F(a)=1$. 
From \eqref{equation_G}, $F(0^+)$ is a solution of   
$l\,F(0^+)^2-(1+l)F(0^+)+1=0.$
If $l=0$, this equation has the unique root $1$. If $l>0$, the two roots are $1$ and $1/l$. But  $1/l>1$ and  obviously $F(0^+)\leq1$, so  that $F(0^+)=1$. 
That ends the proof of the weak convergence of $\tau_n/m_n$.
\end{proof}

\me
\bigskip \noindent We can now proceed with the following proof.

\begin{proof}[Proof of Theorem \ref{Behavior_Tnb}]
Let $Z=\sum_{k\geq0}\alpha(1-\alpha)^kZ_k$ be defined as in the statement of the theorem.
We need the following elementary result which can be proved thanks to a simple induction: for every  $a>0$ and for all complex numbers $z_1,z_2,\dots,z_n,u_1,\dots u_n$ with modulus less than $1$
 \begin{equation}
  \label{prod}
\left|\prod_{i=1}^nz_i-\prod_{i=1}^nu_i\right|\leq\sum_{i=1}^n\left|z_i-u_i\right|.
 \end{equation}
Then, recalling that $T_n=\sum_{k\geq n}\tau_k$ where the $\tau_k$'s are independent,
{\setlength\arraycolsep{2pt}
\begin{eqnarray}&&\left|\Esp_\infty\left(\exp\left(-a\frac{T_n}{\Esp_\infty(T_n)}\right)\right)-\Esp\left(\exp\left(-aZ\right)\right) \right| \nonumber \\
&&\qquad =\left|\prod_{k\geq n}\Esp\left(\exp\left(-a\frac{\tau_{k}}{ \Esp_\infty(T_n)}\right)\right)-\prod_{k\geq0}\Esp\left(\exp\left(-a\alpha(1-\alpha)^kZ_k\right)\right) \right|\nonumber\\
 &&\qquad \leq \sum_{k\geq0}\left|\Esp\left(\exp\left(-a\frac{\tau_{k+n}}{ \Esp_\infty(T_n)}\right)\right)-\Esp\left(\exp\left(-a\alpha(1-\alpha)^kZ_k\right)\right) \right|. \label{eq:1816}
  \end{eqnarray}}
From Lemma \ref{equivalents_alpha_non_nul}, we know that in
$\p_\infty$-distribution,
$\tau_n/ m_n$ converges to $\zeta$. 
Then, thanks to \eqref{equivalents_ratios_alpha} and the fact that $r_n\to\alpha$, we have for $k\geq0$
$$\frac{\tau_{k+n}}{\Esp_\infty(T_n)}=\frac{m_{n+k}}{\Esp_\infty(T_{n+k})}\prod_{i=1}^k\frac{\Esp_\infty[T_{n+i}]}{ \Esp_\infty[T_{n+i-1}]}\cdot\frac{\tau_{k+n}}{m_{n+k}}\,\overset{\textrm{(d)}}{\linf{n}}\,\alpha(1-\alpha)^k\mathcal \zeta.$$
The uniqueness in \eqref{equation_G} ensures that the variables $(Z_k)_k$ are distributed as $\zeta$. Then, with the last display, we get that all the terms of the sum in \eqref{eq:1816} vanish as $n\to+\infty$.
We proceed by bounded convergence. Using that
$1-\exp(-x)\leq x$ for any $x\geq0$,  we get for $k,n\geq0$
 {\setlength\arraycolsep{2pt}
\begin{eqnarray}
&&  \left|\Esp\left(\exp\left(-a\frac{\tau_{k+n}}{ \Esp_\infty(T_n)}\right)\right)-\Esp\left(\exp\left(-a\alpha(1-\alpha)^kZ_k\right)\right) \right|\nonumber \\
&& \qquad \leq
 \left|1-\Esp\left(\exp\left(-a\frac{\tau_{k+n}}{ \Esp_\infty(T_n)}\right)\right)\right|+\left|1-\Esp\left(\exp\left({-a\alpha(1-\alpha)^kZ_k}\right)\right) \right|\nonumber\\
 &&\qquad \leq a\frac{m_{k+n}}{ \Esp_\infty(T_n)}+a\alpha(1-\alpha)^k\Esp(Z_0).\label{1911}
\end{eqnarray}}
By differentiating \eqref{equation_G} at $0$, one finds
$\Esp_\infty[Z_0]=1$. Moreover,
$$\frac{m_{n+k}}{ \Esp_\infty(T_n)}=\frac{ \Esp_\infty(T_{n+1})}{ \Esp_\infty(T_n)}\frac{ \Esp_\infty(T_{n+2})}{ \Esp_\infty(T_{n+1})}\cdots\frac{\Esp_\infty(T_{n+k})}{ \Esp_\infty(T_{n+k-1})}\frac{m_{n+k}}{\Esp_\infty(T_{n+k})}.$$
Since $m_{k+n}/\Esp_\infty[T_{k+n}]\leq 1$ and $\Esp_\infty(T_{n+1})/ \Esp_\infty(T_n)\to1-\alpha<1$ as $n\to+\infty$, there exist  $n_0\in\Nat$, $\beta<1$ and $C>0$ such that
$m_{k+n}/\Esp_\infty(T_n)\leq C\beta^k$ for all $k\geq0, n\geq n_0.$
Thus, coming back to \eqref{1911}, for $n\geq n_0$, we have
$$\left|\Esp\left(\exp\left(-a\frac{\tau_{k+n}}{ \Esp_\infty(T_n)}\right)\right)-
\Esp\left(\exp\left(-a\alpha(1-\alpha)^kZ_k\right)\right)\right|\leq C\beta^k+ a\alpha(1-\alpha)^k.$$
Since the r.h.s. in the last display is summable, the proof is complete.
\end{proof}

\me
\subsection{The gradual regime}
\label{gradualspeed}
\bigskip \noindent 
We now  focus on the second regime and specify the fluctuations of $T_{n}$.  In this case, we will  obtain a weak and a strong law of large numbers.

\begin{thm}\label{Behavior_Tn}
We assume that  \eqref{condition_extinction},  \eqref{cv_{S}}  and \eqref{condition_croissance} hold.    
\begin{enumerate}[{\normalfont(i)}]  
\item If  $\ \displaystyle\frac{\Esp_{n+1}(T_n)}{\Esp_\infty(T_n)}\linf{n}  0$, then
\begin{equation*}\label{CV_Tn_proba}
\frac{T_n}{ \Esp_\infty(T_n)}\linf{n} 1 \qquad \text{ in } \p_\infty-\text{probability}.
\end{equation*}

\item If
 \begin{equation}\label{hypothese_variances}
  \lim_{n\to+\infty}\frac{\Var_{n+1}(T_n)}{\Vinf{T_n}}=0
 \end{equation}
and if
\begin{equation}\label{hypothese_moment_ordre3}
  \lim_{n\to+\infty}\Vinf{T_n}^{-3/2}\sum_{k\geq n}\Esp_{k+1}(\left|T_{k}-\Esp_{k+1}(T_k)\right|^3)=0,
 \end{equation}
we have
\begin{equation*}\label{CLT_T_n}
\frac{T_n-\Einf{T_n}}{\Var_\infty(T_n)^{1/2}}\underset{n\to+\infty}{\overset{\mathrm{(d)}}\longrightarrow} {\cal N},
\end{equation*}
where ${\cal N}$ follows a standard normal distribution.
\end{enumerate}
\end{thm}

\bi Using Cauchy-Schwarz inequality, we note that under \eqref{condition_croissance}, Assumption   \eqref{hypothese_variances}   implies that $\Esp_{n+1}(T_n)/ \Esp_\infty(T_n)\rightarrow 0$, as $n$ tends to infinity. \bi 


\begin{proof}[Proof of Theorem \ref{Behavior_Tn}]{}

\me (i) - We suppose here that $r_n\rightarrow 0$.
Let $\eps>0$. Using Bienaym\'e-Tchebychev inequality and the independence of  the random variables $(\tau_n)_n$, we have
\begin{equation}\label{Bienayme_Tchebychev}
\p_\infty\left(\left|\frac{T_n}{\Esp_\infty(T_n)}-1\right|>\eps\right)\leq\frac{\Var_{\infty}(T_n)}{\eps^2 \Esp_\infty(T_n)^2}=\frac{\sum_{k\geq n}\Var(\tau_k)}{\eps^2 \Esp_\infty(T_n)^2}.
\end{equation}
As  $\Esp_\infty(T_{n+1})/m_n=1/r_{n}-1\rightarrow + \infty$ as $n\to+\infty$,
for all $A>0$, there exists an integer $n_0$ such that, for $n\geq n_0$, $\Esp_\infty(T_{n+1})\geq A\,m_n$ and
$$\Esp_\infty(T_n)^2=\Big(\sum_{k\geq n}m_k\Big)^2\geq2\sum_{k\geq n}m_k\sum_{l>k}m_l\geq2A\sum_{k\geq n}m_k^2,$$
since $\sum_{l>k}m_l = \Esp_\infty(T_{k+1}) \geq A m_{k}$.
Coming back to \eqref{Bienayme_Tchebychev}, for $n\geq n_0$, we have
\begin{equation}\label{eq:1553}
\p_\infty\left(\left|\frac{T_n}{\Esp_\infty(T_n)}-1\right|>\eps\right)\leq \frac{1}{2A\eps^2}\,\frac{\displaystyle\sum_{k\geq n}\Var(\tau_k)}{\displaystyle\sum_{k\geq n}m_k^2}.
\end{equation}
By assumption \eqref{condition_croissance}, there exists $C>0$ such that $\Var(\tau_n)\leq C m_n^2.$
Hence, the r.h.s. of \eqref{eq:1553} goes to $0$ as $A\to+\infty$  and the proof of the convergence in probability is complete.


\me
 (ii) - 
We follow classical  ideas for  the proof of  central limit theorem for partial sums of independent random variables (see  Theorem 27.2 in Billingsley \cite{Billingsley1986}).

\noindent 
We first note that $\Vinf{T_n} \geq \Vinf{T_{n+k}}$ for any $n,k\geq0$ since $T_{n}=\sum_{i\geq n} \tau_i$ and the r.v. $(\tau_i : i\geq 0)$ are independent. Using \eqref{hypothese_variances}, it ensures that
 \begin{equation}
  \label{uniformite}
\sup_{k\geq 0}\frac{\Var(\tau_{k+n})}{\Vinf{T_n}}\leq \sup_{k\geq 0}\frac{\Var(\tau_{k+n})}{\Vinf{T_{n+k}}}\linf{n}0.
 \end{equation}
This convergence being uniform with respect to $k\geq0$, we have
$$\sum_{k\geq0}\log \left(1-\frac{t^2}{2}\frac{\Var(\tau_{k+n})}{\Vinf{T_n}}\right)\equinf{n}-\frac{t^2}{2}\sum_{k\geq0}\frac{\Var(\tau_{k+n})}{\Vinf{T_n}}=-\frac{t^2}{2}.$$
Therefore,
\begin{equation}\label{tric}\,\exp(-t^2/2) = \lim_{n\to \infty} \prod_{k\geq 0}\left(1-\frac{t^2}{2}\frac{\Var(\tau_{k+n})}{\Vinf{T_n}}\right). \end{equation}

\noindent Let us now prove that
$\,\frac{T_n-\Einf{T_n}}{\Vinf{T_n}^{1/2}}\,$ converges in distribution as $n\to+\infty$ toward a standard normal random variable.
 By L\'evy's theorem and \eqref{tric}, it suffices to prove that for any fixed $t$, 
$$U_n=\Einf{\exp\left(it\frac{T_n-\Einf{T_n}}{\Vinf{T_n}^{1/2}}\right)}-\prod_{k\geq 0}\left(1-\frac{t^2}{2}\frac{\Var(\tau_{k+n})}{\Vinf{T_n}}\right)$$
vanishes as $n\to+\infty$. 
 First, since the $\tau_n$'s are independent,  for all $t\in\R, n\geq0$
 \begin{equation}\label{eq1711}
  |U_n|=
  \left|\prod_{k\geq0} \Esp\left( \exp \left(it\frac{\tau_{k+n}-\Esp(\tau_{k+n})}{ \Vinf{T_n}^{1/2}}\right)\right)-\prod_{k\geq 0}\left(1-\frac{t^2}{2}\frac{\Var(\tau_{k+n})}{\Vinf{T_n}}\right)\right|.
 \end{equation}
 According to  \eqref{uniformite}, for $n$ large enough and for any $k$,
 all the factors of the second product of \eqref{eq1711} are less than 1. Hence, thanks to \eqref{prod}, we have the 
 inequality
 \begin{equation}\label{eq1727}
  |U_n|\leq \sum_{k\geq0}\left|\Esp\left(\exp\left(it\frac{\tau_{k+n}-\Esp(\tau_{k+n})}{\Vinf{T_n}^{1/2}}\right)\right)-1+\frac{t^2}{2}\frac{\Var(\tau_{k+n})}{\Vinf{T_n}}\right|.
 \end{equation}
 According to equation (27.11) in \cite[p.369]{Billingsley1986}, for any centered random variable $\xi$ with a finite second moment, we have
 $\left|\EE{\exp(it\xi)}-1+\Var(\xi)t^2/2\right|\leq\EE{\min(|t\xi|^2,|t\xi|^3)}$ for $t\geq 0$.
Using this inequality with the random variables $[\tau_{n+k}-\Esp(\tau_{n+k})]/\Vinf{T_n}^{3/2}$, we obtain from \eqref{eq1727} that
\begin{equation*}\label{eq1745}
  |U_n|\leq |t|^3\sum_{k\geq0}\frac{\Esp\left(\left|\tau_{k+n}-\Esp(\tau_{k+n})\right|^3\right)}{\Vinf{T_n}^{3/2}}.
 \end{equation*}
 and using \eqref{hypothese_moment_ordre3}, $U_n$ goes to $0$ as $n\to+\infty$.
This completes the proof.
\end{proof}

\bi Let us now state a strong law of large numbers. 
\begin{thm}
\label{SLGN}
We assume that  \eqref{condition_extinction},  \eqref{cv_{S}}, \eqref{condition_croissance}  and \eqref{hyp-carre} hold. Then  the sequence $\,(\frac{T_n}{ \Esp_\infty(T_n)})_{n}$ converges to $1$,  $\p_\infty$-almost surely.
\end{thm}

\me
\begin{proof}[Proof of Theorem  \ref{SLGN}]
 We  prove the a.s. convergence when \eqref{hyp-carre} holds. According to the law of large numbers of Proposition 1 in \cite{Klesov83}, we just need to check that
\begin{equation}\label{condition_serie}
 \sum_{n\geq0}\frac{\Var(\tau_{n})}{\Esp_\infty(T_n)^2}<+\infty.
\end{equation}

\noindent Using  $\Var(\tau_{n})\leq \widehat C (\Esp_{n+1}(T_{n}))^2$ thanks to \eqref{condition_croissance} and Assumption \eqref{hyp-carre} ensure  \eqref{condition_serie} and the proof is complete.
\end{proof}

\me Let us illustrate this result with the example $\lambda_{n}=0$ and $\mu_{n}=n\log^\gamma n$ $(\gamma>1)$.  One can check that 
 $$\Esp_{n+1}(T_n) = \frac{1}{(n+1)\log^{\gamma} (n+1)}$$ and
$$\Einf{T_n}=\sum_{k\geq n+1}\frac{1}{k\log^\gamma k}\equinf{n}
\frac{1}{(\gamma-1)\log^{\gamma-1} n}.$$
Using Lemma  \ref{lemme_moments_tau_n} (ii), we know that  \eqref{condition_extinction},  \eqref{cv_{S}}, \eqref{condition_croissance} hold. It's also easy to check that \eqref{hyp-carre} is true. Then we can apply Theorem \ref{SLGN} to get that $\,T_{n}/\Einf{T_n}\,$ converges a.s. to $1$ as $n\to \infty$. 

\bi Other examples will be developed in Section 5.


\subsection{Comments and examples}
\label{exemples}

In the statement of the three previous theorems appear  different assumptions. Let us show here that  the choice of these assumptions is very subtile and illustrate our results.

\me
1- One can exhibit a situation of gradual regime where the assumptions of Theorem \ref{Behavior_Tn} are satisfied, while \eqref{condition_limite} fails. 

\me We assume that for each $n\geq 0$, $\mu_n=n^2$ and 
$$\lambda_n=\frac{n^2}{2} \quad \text{if } n \in \mathbb N-4\mathbb N; \qquad \lambda_n=2n^2 \quad \text{if } n \in 4\mathbb N.$$
Then $$\limsup_{n\rightarrow \infty} \lambda_n/\mu_n >1.$$
For each $n \geq 1$, 
$$\frac{1}{2^{1+n/2}} \leq \lambda_n\pi_n \leq \frac{1}{2^{n/2}}, \quad \frac{1}{n^22^{2+n/2}} \leq \pi_n \leq \frac{2}{n^22^{n/2}},$$ 
 so that   \eqref{condition_extinction} is satisfied
 and \eqref{temps} gives
$$\Esp_{n+1}(T_n)={\cal O}\left(\frac{1}{n^2}\right).$$ 
 Then \eqref{cv_{S}} also holds,
$\Esp_{i+1}(T_i)/\Esp_{n+1}(T_n)$ is bounded for $i\geq n$, and \eqref{condition_croissancebis} can be easily checked since $\frac{\lambda_{i}\pi_{i}}{\lambda_{n}\pi_{n}}\leq \frac{2}{2^{(i-n)/2}}$. Then \eqref{condition_croissance} is also fulfilled.
Thus the assumptions of Theorem \ref{Behavior_Tn} are satisfied.

\bi
2 - 
The assumptions for the weak  law of large numbers in Theorem \ref{Behavior_Tn} are not sufficient to obtain the strong law of large numbers (Theorem \ref{SLGN}). 
Let us consider a pure death process with  $\mu_n=\exp(n/\log n)\log n$ and prove that the convergence holds in probability but not almost surely.

\me
Here $l=0$ and 
$$\Esp(\tau_n)\underset{n\to\infty}= \frac{1}{\mu_{n+1}}, \qquad \Einf{T_n}= s_{n+1}=\sum_{k\geq n+1}\frac{1}{\mu_k}.$$ Moreover, as
$\mu_n$ is non-decreasing,
\begin{equation}
\int_n^\infty\frac{e^{-x/\log(x)}}{\log x}\dif x\leq  s_{n}\leq \int_n^\infty\frac{e^{-x/\log(x)}}{\log x}\dif x+\frac{e^{-n/\log(n)}}{\log n} \nonumber
 \end{equation}
and
$$\int_n^\infty\frac{e^{-x/\log(x)}}{\log x}\dif x\equinf{n}\int_n^\infty\left(\frac{1}{\log x}+\frac{1}{(\log x)^2}\right)e^{-x/\log(x)}\dif x=e^{-n/\log(n)}.$$
Combining the two last displays and recalling $r_n=\Esp(\tau_n)/\Einf{T_n}$, we have
$$ s_{n}\sim \exp(-n/\log n), \qquad
r_n\sim 1/\log n, \qquad r_n\to 0,$$
so that   $T_n/s_{n+1}$ goes to $1$ in probability. \\
 We prove now that 
 the almost sure convergence does not hold and proceed by contradiction. Thus,
we  assume now  that  $V_n:=T_n/ s_{n+1}$ does converge a.s. toward $1$. We have
\begin{equation*}\label{eq:1658}
V_{n+1}-V_n=V_{n+1}\left(1-\frac{ s_{n+2}}{ s_{n+1}}\right)-\frac{\tau_{n}}{ s_{n+1}}.
  \end{equation*}
By hypothesis, the left hand side of the latter a.s. vanishes as $n\to+\infty$. Moreover, simple computations lead to $s_{n+1}/s_{n}\to1$ and the first term in
the r.h.s. of the last display a.s. goes to $0$ since our assumption implies that a.s. $(V_n)_n$ is bounded.
Hence, putting all pieces together, the term $\tau_{n}/s_{n+1}$  has to go to $0$ a.s.\\
To get a contradiction thanks to Borel-Cantelli's lemma, it suffices to prove that for $\eps$ small enough,
$$\sum_{n\geq0}\p(\tau_{n} / s_{n+1}>\eps)=\infty,$$
recalling  that  the random variables $\tau_n$ are independent.
The law of   $\tau_{n}$  is    exponential  with parameter $\mu_{n+1}$.
Then,
$\p(\tau_{n}/s_{n+1}>\eps) = \exp(-\eps\mu_{n+1} s_{n+1})$. Since  $\mu_n\,s_{n}\sim\log n$ as $n\to+\infty$, there exists $C>0$ such that
$$\p(\tau_{n}/ s_{n+1}>\eps)\geq e^{-\eps C\log n}=\frac{1}{n^{C\eps}},$$
which completes the proof since  $\sum_{n\geq0}\p\left(\tau_{n+1}/ s_{n}>\eps\right)$ is infinite as soon as $\eps$ is small enough.

\bi
3 - 
In Theorems \ref{Behavior_Tnb} and \ref{Behavior_Tn}, we did not consider the case where $r_n=\Esp(\tau_n)/\Einf{T_n}$ does not converge. In such a case, one can only state  analogous results along the convergent
subsequences.
For instance, if $\mu_{2n}=\mu_{2n+1}=3^{2n}$, we have
$$r_{2n}\linf{n}\frac{4}{9} \quad \textrm{ and } \quad r_{2n+1}\linf{n}\frac{4}{5}.$$
Theorem \ref{Behavior_Tnb} then still holds but the subsequences $(T_{2n}/\Einf{T_{2n}})_n$ and $(T_{2n+1}/\Einf{T_{2n+1}})_n$  converge in distribution to different limits.

\medskip \noindent One can also find examples where $0=\liminf_n r_n<\limsup_n r_n$. Then, $(T_n)_n$ has two subsequences satisfying the two regimes of Theorems \ref{Behavior_Tnb} and \ref{Behavior_Tn}.

\section{Speed of coming down from infinity}\label{section_comportement_X}
\label{SCDI}

In this section, we use the asymptotic  behavior of $T_n/\Esp_{\infty}(T_n)$ obtained
in the previous Section to derive the short time behavior of $X(t)$.  We prove that 
$X$ behaves as the following non-increasing function that tends to infinity as $t\to0$
\begin{equation*}
\label{defv}
v(t):=\inf\{n\geq0;\ \Esp_{\infty}(T_n)\leq t\}.
\end{equation*}
The function $v$ is a c\`ad-l\`ag step function defined on $\mathbb{R}_{+}$, decreasing from $+\infty$ to $0$, it equals  $n$ between $  \Esp_{\infty}(T_n)$ and $ \Esp_{\infty}(T_{n-1})$, and  $0$ after $ \Esp_{\infty}(T_0)$.

\smallskip \noindent The short time behavior of $X$  relies on the inversion of the asymptotic behavior of $T_n$ (see forthcoming Lemma \ref{inv}) 
and the control of the excursion of $X$ between two successive stopping times
$[T_{n+1},T_n]$ (see forthcoming Lemma \ref{excursion}). The latter is true under  the assumption
\be \label{exc}\limsup_{n\rightarrow \infty}  \frac{\lambda_n}{\mu_n} <1. \ee
This assumption is already necessary for   Theorem \ref{Behavior_Tnb} but not for  Theorem \ref{Behavior_Tn}, as developed in Subsection 3.3 Example 1. 

\me The proof is organized as follows.  We introduce the a.s. non-increasing process $Y$ defined by
$$Y(t)=n \quad \text{if} \quad t\in [T_{n},T_{n-1})$$
In the next Section, we prove  that this (more regular) process comes down from infinity at speed $v(t)$ and
we compare the processes  $X(t)$  and $Y(t)$ as $t\rightarrow 0$ by the study of the height of the excursions  of the process $X$. 

\subsection{Height of the excursions  and  non-increasing process.}

\me
We first compare the processes  $X(t)$  and $Y(t)$   by estimating 
the number of birth events between the times $T_n$ and $T_{n-1}$:
$$H_n=\#\{  s  \in [T_n,T_{n-1}) : X(s)-X(s-)>0\}, \quad n\geq1.$$

\begin{lem} \label{excursion} (i) We have  
 \begin{equation}\label{encadrement_X_Y}
0\leq\frac{X(t)}{v(t)}-\frac{Y(t)}{v(t)}\leq \frac{H_{Y(t)}}{Y(t)} \frac{Y(t)}{v(t)}
\end{equation}

(ii) Under Assumption $(\ref{exc})$, 
 $\displaystyle\frac{H_n}{n}\rightarrow 0$ $\ \p_{\infty}$ a.s. 
\end{lem}

\begin{proof}
(i) For any $t \in [T_n,T_{n-1})$, $Y(t)=n$ and
$ 0\leq X(t)-Y(t) \leq H_n$, so the first part is obvious.

\noindent (ii)  Let us first notice that $H_n$ equals the number of positive jumps between time $T_n$ and $T_{n-1}$ of a random walk whose transition probabilities are given by
$p_{i,i+1}=\lambda_i/(\lambda_i+\mu_i)$, $ p_{i,i-1}=\mu_i/(\lambda_i+\mu_i)$ for $i\geq 1.$
Using \eqref{exc}, we can choose $n_0$ large enough so that
$\,p=\sup_{n\geq n_0}\lambda_n/(\lambda_n+\mu_n) <1/2.$
Then, for $n\geq n_0$, $H_n$ is stochastically dominated by $T$, the hitting time of $n-1$ by a simple random walk starting at $n$, with probability transitions $(1-p,p)$.
Since $p<1/2$, $\Esp(T^2)<+\infty$. Hence $\sup_{n\geq n_0}\Einf{H_n^2}<+\infty$
  and  the sequences $(\Einf{H_n})_n$ and $(\Einf{H_n^2})_n$ are bounded. 
  
  \smallskip \noindent Let us now consider  the Laplace transform of $H_n$  given by $\widehat G_n(a)=\Einf{\exp(-aH_n)}$. In the same vein as we have obtained \eqref{recurrence_Gn} and  by applying the strong Markov property at the first time when $X$ jumps after $T_n$, we get the recursion formula
\begin{equation}\label{recurrence_hat_Gn}
\widehat G_n(a)=\frac{\mu_n}{\lambda_n+\mu_n}+\frac{\lambda_n}{\lambda_n+\mu_n}e^{-a}\widehat G_n(a)\widehat G_{n+1}(a),\quad a\geq 0,\ n\geq1.
\end{equation}
Differentiating \eqref{recurrence_hat_Gn} twice at $a=0$, the second moment of $H_n$ satisfies the following recursion formula
\begin{equation*}
\frac{\mu_n}{\lambda_n}\Einf{H_n^2}=\Einf{H_{n+1}^2}+1+2\left(\Einf{H_n}
+\Einf{H_{n+1}}+\Einf{H_n}
\Einf{H_{n+1}}\right).\label{recurrence_Hn2}
\end{equation*}
 We have seen  that the right hand side of the latter is uniformly bounded in $n\geq0$.  It entails that there is $C>0$ such that
\begin{equation}\label{majoration_moment2_Hn}
\Einf{H_n^2}\leq C\,\frac{\lambda_n}{\mu_n},\quad n\geq 1.
\end{equation}
Finally,
$\Einf{\sum_{n\geq1}\left(\frac{H_n}{n}\right)^2}\leq C\,\sum_{n\geq1}\frac{1}{n^2}\frac{\lambda_n}{\mu_n}<+\infty$ using again \eqref{exc}.
%
In particular,  it turns out that the sequence $\,(\frac{H_n}{n})_{n}\,$ almost surely goes to $0$ as $n\to+\infty$.
\end{proof}

\bi
Let us now introduce  the quantity
$$R(x,y):=\Esp_{\infty}(T_{[ x]})/\Esp_{\infty}(T_{[ y]})$$
and  study the behavior of $Y(t)/v(t)$ as $t$ tends to $0$. 

\me
\begin{prop} \label{inv}
\emph{(i)}  If  $T_n/\Esp_{\infty} (T_n)$ is tight on $(0,\infty)$ and for every $x>1$, 
$\lim_{n\rightarrow \infty} R(nx,n)=0$ 
then $Y(t)/v(t)\rightarrow 1$ in $\p_\infty-\textrm{probability}$. 

\me
\emph{(ii)} 
 If  $T_n/\Esp_{\infty} (T_n) \to 1$ in $\p_\infty-\textrm{probability}$ and if for every $x>1$,  $\limsup_{n\rightarrow \infty} R(nx,n)<1$,
then $Y(t)/v(t)\rightarrow 1$ in $\p_\infty-\textrm{probability}$. 

\me
\emph{(iii)}  If $\lim_{n\rightarrow \infty}R(n+1,n)=1$ and $\limsup_{n\rightarrow \infty} R(nx,n)<1$   for every $x>1$,
then 
\be
\label{*}\lim_{a \rightarrow 1} \limsup_{t\rightarrow 0} |\frac{v(a t)}{v(t)}-1|=0.\ee
 If additionally  $T_n/\Esp_{\infty}(T_n)\rightarrow 1$ a.s., 
then $Y(t)/v(t)\rightarrow 1$  $\p_{\infty}$ a.s. 
\end{prop}

\begin{proof}  (i) Under the tightness assumption,  for any $\epsilon >0$, there exist $0<A\leq B$ such that for every $n\geq 0$,
\begin{equation}\label{eq_1520}
\p_\infty(A\leq  T_n/ \Esp_\infty(T_n)  \leq B)\geq 1-\epsilon.
\end{equation}
Moreover, for every $x>1$ and  $n$ large enough,
\begin{equation}\label{eq_1521}
R(nx,n) 
\leq \min(1/(2B),A/2).
\end{equation}
By the definition of the function $v$, we have
$$\Esp_\infty\left(T_{v(t)}\right)\leq t<\Esp_\infty\left(T_{v(t)-1}\right).$$
It implies that for any $t>0$,
$$\pinf{T_{[xv(t)]}\leq \frac{t}{2}}\geq \pinf{T_{[xv(t)]} \leq  \frac{\Esp_\infty\left(T_{v(t)}\right)}{2}}=\pinf{\frac{T_{[xv(t)]}}{ \Esp_\infty\left(T_{[xv(t)]}\right)} \leq  \frac{R(v(t),xv(t))}{2}}.$$
Hence, using \eqref{eq_1520} and \eqref{eq_1521} and for $t$ small enough,
$$\p_\infty(T_{[xv(t)]}\leq t/2)\geq\p_\infty \left(\frac{T_{[xv(t)]}}{ \Esp_\infty\left(T_{[xv(t)]}\right)}   \leq B  \right )\geq 1-\epsilon.$$
We similarly get that for $t$ small enough
$$\p_\infty(T_{[v(t)/x]}\geq 2t)\geq\p_\infty \left(\frac{T_{[v(t)/x]}}{ \Esp_\infty\left(T_{[v(t)/x]}\right)}   \geq A \right )\geq 1-\epsilon.$$
Then,  we have for $t$ small enough
$$\p_\infty(T_{[xv(t)]}\leq t/2, T_{[v(t)/x]}\geq 2t)\geq 1-2\epsilon.$$
Since $Y$ is non-increasing, that implies 
$\p_\infty(   Y(t) \in [v(t)/x,v(t)x])\geq 1-2\epsilon$
and ensures that
$Y(t)/v(t)$ tends to $1$ in probability as $t\rightarrow 0$. 

\me The proof of (ii) follows the same steps as the one of (i). Since $T_n/\Esp_{\infty} (T_n) \to 1$ in $\p_\infty-\textrm{probability}$ as $n$ tends to infinity,  we can  choose for any $x>1$, $A$ and
 $B$ close enough to $1$ and $a<1$ such that
$$\p_\infty \left(A\leq \frac{T_{[v(t)/x]}}{\Esp_\infty\left(T_{[v(t)/x]}\right)}   \leq B \right )\geq 1-\epsilon, \quad R(nx,n) 
\leq a\min(1/B,A)$$
for $t$ small enough and $n$ large enough.
We   conclude as previously using now $\p_\infty(T_{[xv(t)]}\leq a t, T_{[v(t)/x]}\geq t/a)\geq 1-2\epsilon.$

\me Let us now prove (iii).  We first note that  $v$ is non-increasing.  Let us first prove that $$\lim_{a \downarrow 1} \liminf_{t\rightarrow 0} \frac{v(a t)}{v(t)}=1.$$
Then
for all $t>0$ and $a>1$,
\be
\label{quotientv}
\frac{v(at)}{v(t)} \leq 1.
\ee
Moreover,
$$\Esp_{\infty}(T_{v( at)})\leq t a, \qquad \Esp_{\infty}(T_{v(t)})\,R(v(t)-1,v(t))=\Esp_{\infty}(T_{v(t)-1})\geq t.$$
For all $\eta>1$, the first assumption of (iii) yields  $ R(v(t)-1,v(t))\leq \eta$  for  $t$ small enough and we get
$$\Esp_{\infty}(T_{v(at)})\leq a\, \eta \Esp_{\infty}(T_{v(t)}),$$
which implies that
$$\lim_{a \downarrow 1} \limsup_{t\rightarrow 0} R(v(a t),v(t))\leq 1.$$
Moreover  the second  assumption of (iii) ensures that  $\ \liminf_{n\rightarrow \infty} R(nx,n)>1\,$
   for every $x<1$.
We add that for each  $t$ such that $v(at) \leq xv(t)$, we have
$$R(v(at),v(t))\geq R(xv(t),v(t))$$
Combining the three last displays ensures that for any $x<1$ ensures that
 $$ \lim_{a \downarrow 1} \liminf_{t\rightarrow 0} \frac{v(a t)}{v(t)} \geq x.$$
Letting $x \rightarrow 1$ and recalling (\ref{quotientv}) yields $$\ \lim_{a \downarrow 1} \limsup_{t\rightarrow 0} \big\vert \frac{v(a t)}{v(t)}-1\big\vert=0.$$
To conclude to  the first part of (iii), it remains to consider $a<1$, which is simply derived from the previous limit by changing $t$ into $t/a$. \\
We assume now that  $T_n/\Esp_{\infty}(T_n)\rightarrow 1$ a.s. and \eqref{*} enables us to compose the equivalence by $v$:
$$\lim_{n\rightarrow \infty} \frac{v(T_n)}{v(\Esp_{\infty}(T_n))}= 1=\lim_{n\rightarrow \infty} \frac{v(T_n)}{n} \quad  a.s.$$
since  $v(\Esp_{\infty}(T_n))=n$ by definition of $v$ and  $n\rightarrow \Esp_{\infty}(T_n)$ is decreasing.
Noting that $T_n$ goes to $0$ a.s. and 
  $$ \frac{n}{v(T_{n})} \leq  \frac{Y(t)}{v(t)}\leq \frac{n}{v(T_{n-1})}$$
a.s. on the event  $t\in [T_{n}, T_{n-1})$ ends up the proof of (iii).
\end{proof}

\me

\subsection{Fast coming down from infinity}
\label{speedfast}

We state the convergence  in probability  inherited from Theorem \ref{Behavior_Tnb} ($\alpha>0$).  

\begin{thm}\label{cor-v}  We assume that   \eqref{condition_extinction}, \eqref{cv_{S}}  and  \eqref{condition_limite}  hold and
 $$\Esp_{n+1}(T_n)/\Esp_\infty(T_n)\linf{n}  \alpha \in (0,1].$$
Then
$$\lim_{t\to0}\frac{X(t)}{v(t)}=1 \textrm{ in }\p_\infty-\textrm{probability}.$$
 \end{thm}

\begin{proof}
Under the assumptions of Theorem 
 \ref{Behavior_Tnb}, 
 the sequence  $T_n/\Esp_{\infty} (T_n)$ converges in law to a random variable whose law is supported on $(0,\infty)$, since $G(a) \rightarrow 0$ as $a\rightarrow \infty$. Then this sequence is tight on $(0,\infty)$. Let us fix $x>1$ and show that 
$$\lim_{n\rightarrow \infty} R(nx,n)=\lim_{n\rightarrow \infty} \frac{\Esp_{\infty}(T_{[ nx]})}{\Esp_{\infty}(T_{ n})} = 0.$$ 
Indeed,  as seen in \eqref{equivalents_ratios_alpha}, 
$\lim_{n\to+\infty}\frac{\Esp_\infty(T_{n+1})}{\Esp_\infty(T_n)}=1-\alpha$.
Then for $n$ large enough,
$$\lim_{n\rightarrow \infty}\frac{\Esp_{\infty}(T_{[ nx]})}{\Esp_{\infty}(T_{ n})} = \lim_{n\rightarrow \infty}\prod_{j=n}^{[nx]-1}\frac{\Esp_\infty(T_{j+1})}{\Esp_\infty(T_j)}= \lim_{n\rightarrow \infty}(1-\alpha)^{[nx]-n} = 0.$$
Then by Proposition \ref{inv}-(i), $Y(t)/v(t)\rightarrow 1$ in $\p_\infty-\textrm{probability}$. 
Lemma   \ref{excursion} will then allow to conclude. 
\end{proof}

\bigskip \noindent
\emph{Example.} Let us consider a pure death process with $\mu_n=e^{\beta n}$ with $\beta>0$ and thus  $\Einf{T_n}\sim e^{-\beta(n+1)}/(1-e^{-\beta})$. Hence, the conditions of Theorem \ref{Behavior_Tnb} hold true with $\alpha=1-e^{-\beta}$ and
$$e^{\beta(n+1)} T_n\overset{\textrm{(d)}}{\linf{n}}\sum_{k\geq0}e^{-\beta k}E_k$$ where the $E_k$'s are independant  exponential random variables  with parameter $1$.
In that case, we can explicitly determine the speed $v$ of Theorem \ref{cor-v} and  we get
$X(t)\sim-(\log t)/\beta$ as $t\to0$ in probability.


\subsection{Gradual  coming down from infinity}

We give now the speed of convergence  and describe the fluctuations of $X$ in the case $\alpha=0$.

\begin{thm} \label{Comportement_X_variation_reguliere} 
 We assume that  \eqref{condition_extinction},  \eqref{cv_{S}},  \eqref{condition_croissance}  hold and $\Esp_{n+1}(T_n)/ \Esp_\infty(T_n)\linf{n}  0$. \\
We assume  also that \eqref{exc} holds and that  for every $x>1$,
\begin{equation}
 \label{condinvv}
\limsup_{n\rightarrow \infty} \frac{\Esp_{\infty}(T_{[nx]})}{\Esp_{\infty}(T_{n})}<1.
\end{equation}
Then,
\begin{equation*}\label{equivalent_Xt_2}
 \lim_{t\to0}\frac{X(t)}{v(t)}=1 \qquad  \text{in } \ \p_\infty-\text{probability}.
\end{equation*}
Assuming further that $\sum_{n}\Big( \Esp_{n+1}(T_n)/ \Esp_\infty(T_n)\Big)^2<\infty$, this convergence holds a.s.
\end{thm}

\me The additional assumption \eqref{condinvv} is required for the inversion (Proposition \ref{inv} (ii)). Indeed the quantity
$\Esp_{\infty}(T_{n})$   should not tend too slowly to $0$. The  example $\mu_n=n\log^{\gamma}n$ ($\gamma>1$) and $\lambda_n=0$ shows that \eqref{condinvv} may fail while   the other assumptions hold (see Section \ref{gradualspeed} for details).

\begin{proof} Let us deal with the convergence in probability and work under $\p_{\infty}$.
 The first four assumptions allow us to apply Theorem \ref{Behavior_Tn}, so that $T_n/\Esp_{\infty}(T_n) \rightarrow 1$ in probability as $n\rightarrow \infty$.
Using \eqref{condinvv}, we can apply Proposition \ref{inv} (ii) to get that $Y(t)/v(t)\rightarrow 1$ in probability.  Moreover Assumption \eqref{exc} enables us to use Lemma \ref{excursion}.
It    ensures that $H_{Y(t)}/Y(t) \rightarrow 0$ a.s. since $Y(t)\rightarrow +\infty \ $ a.s as $t\to 0$ and $X(t)/v(t)-Y(t)/v(t)\rightarrow 0$ in probability and then  the 
convergence   in probability of 
$X(t)/v(t)$ to $1$. 

\me We note now  that 
$$R(n+1,n)= 1-\frac{\Esp_{n+1}(T_n)}{\Esp_{\infty}(T_n)}.$$ Thus,  $\Esp_{n+1}(T_n)/ \Esp_\infty(T_n)\rightarrow  0$ yields 
$R(n+1,n) \rightarrow  1$ as $n\rightarrow \infty$. Then the a.s. convergence is obtained similarly combining Theorem \ref{SLGN}, Lemma \ref{excursion} and Proposition \ref{inv} (iii). 
\end{proof}

\begin{rem} We  remark that if  $\lambda_n=0$ and $\mu_n=n(n-1)/2$, $X(t)$ is the number of blocks of the  Kingman coalescent at time $t$.  In this case,
$$\Esp_\infty(T_n) = \sum_{i\geq n}\frac{2}{i(i+1)} = \frac{2}{n}.$$
Then, 
$v(t) = \frac{2}{t}$ and we recover  from Theorem \ref{Comportement_X_variation_reguliere}  the speed of coming down from infinity for this process, obtained by Aldous in \cite{Aldous99}-paragraph 4.2.:  $\, tX(t)\,\underset{t\to 0}{\longrightarrow}\,2\as.$  \end{rem}

\me
We refer to the next section for more general examples, where we also provide the  fluctuations of $X$ under $\mathbb{P}_{\infty}$ for $t$ close to $0$ using the following result.


\me
\begin{prop} \label{corTLC}We assume that  \eqref{condition_extinction},  \eqref{cv_{S}}  and \eqref{condition_croissance} hold 
and  that
 \eqref{hypothese_variances} and \eqref{hypothese_moment_ordre3} 
 hold. \\
We also assume that
$ \sum_{n\in \mathbb N} \frac{1}{n} \frac{\lambda_n}{\mu_n}<\infty$ and $\eqref{condinvv}$  and that for every  $x \in \R$,
 \begin{equation}\label{eq:1747b}
\frac{t-\Einf{T_{s(x,t)}}}{\sqrt{\Vinf{T_{s(x,t)}}}}\underset{t\to0}\longrightarrow x,
\end{equation}
 where  $s(x,t)=[v(t)+x\sqrt{v(t)}]$. Then, 
\begin{equation}\label{eq:1844}
\sqrt{v(t)}\left(\frac{X(t)}{v(t)}-1\right)\underset{t\to0}{\overset{\mathrm{(d)}}\longrightarrow} {\cal N},
\end{equation}
where ${\cal N}$ follows a standard normal distribution.
\end{prop}
\begin{proof}[Proof of Proposition \ref{corTLC}]
The proof follows the same steps than the previous theorem. We
 use the C.L.T theorem for $T_n$ to  firstly establish a central limit theorem for the a.s. non-increasing  process $Y$.

\me
As  $Y$ is non-increasing, we can follow the proof of the central limit theorem for   renewal processes (as suggested by Aldous for Kingman's coalescent, cf. \cite{Aldous99}). More precisely, for any $t>0, x\in \mathbb R$, we use $\pinf{Y(t)> s(x,t)}=\pinf{T_{s(x,t)}> t}$ to get
$$
\pinf{\sqrt{v(t)}\left(\frac{Y(t)}{v(t)}-1\right)> x}=\pinf{\widetilde Z_{ s(x,t)}> \frac{t-\Einf{T_{s(x,t)}}}{\sqrt{\Vinf{T_{s(x,t)}}}}},
$$
where
$$\widetilde Z_n=\frac{T_n-\Einf{T_n}}{\sqrt{\Vinf{T_n}}} $$
All the assumptions of Theorem \ref{Behavior_Tn} (ii) are met, so $\widetilde Z_n$ converges weakly to a standard normal variable.
Using  \eqref{eq:1747b}, we obtain the C.L.T. \eqref{eq:1844} for $Y$.

 \me We end the proof by deducing the C.L.T for $X$ thanks to the  decomposition  
\begin{equation}
\label{dec}
\sqrt{v(t)}\left(\frac{X(t)}{v(t)}-1\right)=\sqrt{v(t)}\left(\frac{Y(t)}{v(t)}-1\right)+\frac{X(t)-Y(t)}{\sqrt{v(t)}}.
\end{equation}
From \eqref{encadrement_X_Y}, we almost surely have
\begin{equation}\label{encadrement_X_Y_2}
0\leq \frac{X(t)-Y(t)}{\sqrt{v(t)}}\leq \frac{H_{Y(t)}}{\sqrt{Y(t)}}\frac{\sqrt{Y(t)}}{\sqrt{v(t)}}.
\end{equation}
\noindent Using \eqref{majoration_moment2_Hn}, there exists $C$ such that
$$\Einf{\sum_{n\geq1}\left(\frac{H_n}{\sqrt{n}}\right)^2}\leq C\sum_{n\geq1}\frac{1}{n}\frac{\lambda_n}{\mu_n}.$$
Since this series converges by hypothesis, $H_n/\sqrt{n}$ a.s.  goes to $0$ as $n\to+\infty$. Recalling  that \eqref{condition_croissance} and \eqref{hypothese_variances} ensure that
$\Esp_{n+1}(T_n)/\Esp_{\infty}(T_n) \rightarrow 0$, 
all the assumptions of Theorem  \ref{Comportement_X_variation_reguliere} are fulfilled,  so
 $Y(t) \sim X(t) \sim v(t)$ as $t\to0$ in probability. Then  the right hand side of \eqref{encadrement_X_Y_2} vanishes as $t\to0$ in probability  and \eqref{dec} allows us to derive \eqref{eq:1844} from the C.L.T for $Y$ established above.
\end{proof}

 \section {Application for   regularly varying death rates} 
\label{ApplVR}
As mentioned  before (see in particular the end of  Section \ref{tract}),
the following class of  birth and death processes is particulary relevant for population dynamics and population genetics models:
\be
\label{cas-par} \lambda_n\leq Cn, \qquad \mu_n=n^\rho\log^\gamma n, \ \hbox{where}\ C>0,  \rho>1 , \gamma \in \R.\ee

\bi This  is a particular case of a main class that we can attain with our results. In what follows, we will suppose that the birth rate is sub-linear as assumed in \eqref{cas-par}.
 The main assumption is that the death rate varies regularly. Our previous theorems apply in this general context. 

\me
Recall that  a
  sequence of real non-zero numbers $(u_n)_{n\geq0}$ varies regularly with index $\rho \ne 0$ if  for all $a>0$,
 $$\lim_{n\to +\infty}\frac{u_{[an]}}{u_n}=a^{\rho}.$$
  A function $g:[0,+\infty)\to(0,+\infty)$ \emph{varies regularly} at $0$ with index $\rho\ne 0$ if for all $a>0$,
$$\lim_{x\to 0}\frac{g(ax)}{g(x)}=a^\rho.$$

\me

\begin{thm} \label{corVR}
Suppose that $\lim_{n\to+\infty}\lambda_n/\mu_n=0$ and that $(\mu_n)_n$ varies regularly with index $\rho>1$.
Then, 
\be
\label{CVps-VR}
\lim_{n\rightarrow \infty} \frac{\mu_{n+1}(\rho - 1)}{n}T_n=\lim_{t\to0}\frac{X(t)}{v(t)}=1 \qquad \p_\infty-\text{a.s}.\ee
where $v$ is  regularly varying at $0$ with index $1/(1-\rho)$.  Further,
\begin{equation*}\label{CLT_{N-VR}}
\frac{\sqrt{2\rho - 1}}{\sqrt{n}}\Big(T_n-\frac{n}{\mu_{n+1} (\rho -1)}\Big)\underset{n\to+\infty}{\overset{\mathrm{(d)}}\longrightarrow} {\cal N},
\end{equation*}
 Assuming further that $\sum_{n\geq1}\frac{1}{n}\frac{\lambda_n}{\mu_{n}}<+\infty$, we also get  that 
$$\sqrt{(2\rho-1)v(t)}\left(\frac{X(t)}{v(t)}-1\right)$$ converges in law, as $t$ tends to $0$, to a standard normal distribution.
 \end{thm}
 
\bi We recover the central limit theorem for the Kingman coalescent. We also mention that \cite{LT} provides Gaussian limits for more general $\Lambda$ coalescent processes whose ``Kingman part'' is non trivial.

\me
\begin{cor} 
\label{nrho}Assume \eqref{cas-par}. 
Then, 
$$
\lim_{n\rightarrow \infty} {T_n\,(\rho-1)n^{\rho-1}\log^\gamma n}=1 \qquad \p_\infty-\text{a.s}.$$
and
$$\frac{\sqrt{2\rho - 1}}{\sqrt{n}}\Big(T_n-\frac{1}{ (\rho -1) n^{\rho-1}\log^\gamma n}\Big)\overset{\textrm{(d)}}{\linf{n}} {\cal N}$$
where ${\cal N}$ follows a standard normal distribution.
\end{cor}

\bi
\begin{proof}[Proof of Theorem \ref{corVR}] 
Let us first  remark that since  $(\mu_n)_n$ varies regularly with index $\rho>1$, then  $(1/\mu_n)_n$ (resp. $(1/\mu_n^2)_n$) varies regularly with index $-\rho<-1$ (resp. $-2 \rho<-2$). Under the assumptions of Theorem \ref{corVR}, we  note that conditions \eqref{condtech} are satisfied. Indeed,   Theorem 1.5.3 in \cite{Bingham1989} shows that the sequence $(\mu_{n})$ is equivalent to a non-decreasing sequence and Lemma \ref{VR_restes_sommes} applies directly to $\sum1/\mu_n$. Lemma \ref{lemme_moments_tau_n} (ii) can be applied, so that
 \eqref{condition_extinction}, \eqref{cv_{S}},
 \eqref{condition_limite} and \eqref{condition_croissance} are satisfied and 
\be
\label{equivv}
\Esp_{n+1}(T_n^k)\equinf{n}\frac{k!}{\mu^k_{n+1}},\quad \hbox{ for }\ k=1,2,3.
\ee

\me 
Then by forthcoming Lemma \ref{VR_restes_sommes}, 
\begin{equation*}
\Esp_{\infty}(T_n) \equinf{n} \frac{n}{\mu_{n+1} (\rho -1)}. 
\end{equation*}

\noindent Thus  for $x>1$,
\begin{equation*}
 \label{condinv}
\limsup_{n\rightarrow \infty} \frac{\Esp_{\infty}(T_{[nx]})}{\Esp_{\infty}(T_{n})}=x^{1-\rho}.
\end{equation*}
Since $1-\rho<0$, then $x^{1-\rho}<1$. Moreover,
$\sum_{n}\left( \Esp_{n+1}(T_n)/\Esp_{\infty}(T_n)\right)^2$ converges and the assumptions for
 Theorem \ref{SLGN} and Theorem \ref{Comportement_X_variation_reguliere} are satisfied, implying \eqref{CVps-VR}. 

 \smallskip \noindent  To prove the C.L.T. for $(T_n)$, we use  again \eqref{equivv} for $k=1,2$ to get
 $\Var(\tau_n)\sim 1/\mu_{n+1}^2$ as $n\to+\infty$, which implies that
  $(\Var(\tau_n))_n$ varies regularly with index $-2\rho$.  Then   $\Var_{{\infty}}(T_n)=\sum_{i=n+1}^{\infty} \Var(\tau_n)$ and forthcoming Lemma \ref{VR_restes_sommes} ensures that $\Vinf{T_n}$ varies regularly with index $1-2\rho$ and 
 \be
 \label{varTn}
 \Vinf{T_n} \sim \frac{n}{(2\rho-1)\mu_{n+1}^2}.
 \ee
 Therefore we have
 $\,\frac{\Var(\tau_n)}{\Var_\infty(T_n)}\equinf{n}\frac{2\rho-1}{n}$,
 which entails \eqref{hypothese_variances}. \\
\noindent
Moreover, by the triangle  inequality and the binomial theorem, we have
$$\Esp_{k+1}(\left|T_{k}-\Esp_{k+1}(T_k)\right|^3)\leq\Esp_{k+1}(T_k^3)+3\Esp_{k+1}(T_k)\Esp_{k+1}(T_k^2)+4\Esp_{k+1}(T_k)^3.$$
Thanks to \eqref{equivv}, all the terms of the r.h.s. are of order of magnitude $1/\mu_{n+1}^3$ as $n\to+\infty$.
 Thus, using  again Lemma \ref{VR_restes_sommes} and \eqref{varTn}, there is a positive constant $C'$ such that
 $$ \frac{\sum_{k\geq n}\Esp_{k+1}(\left|T_{k}-\Esp_{k+1}(T_k)\right|^3)}{\Vinf{T_n}^{3/2}}\leq \frac{C'}{ \sqrt{n}}.$$
This latter vanishes as $n\to+\infty$ and  \eqref{hypothese_moment_ordre3} is satisfied. 
Hence we apply Theorem \ref{Behavior_Tn} and 
\begin{equation*}
\frac{\sqrt{2\rho - 1}}{\sqrt{n}}\Big(T_n-\frac{n}{\mu_{n+1} (\rho -1)}\Big)\underset{n\to+\infty}{\overset{\mathrm{(d)}}\longrightarrow} {\cal N},
\end{equation*}
where ${\cal N}$ follows a standard normal distribution.

\me Let us now prove the last assertion of Theorem \ref{corVR}. 
To apply Proposition  \ref{corTLC}, we need to prove that
\begin{equation}\label{eq:1747}
\frac{t-\Einf{T_{s(x,t)}}}{\sqrt{\Vinf{T_{s(x,t)}}}}\underset{t\to0}\longrightarrow x\sqrt{2\rho-1}
\end{equation}
and we  first consider the case $x>0$. We note that
\begin{equation}\label{encadrement_1515}
\Einf{T_{v(t)}}-\Einf{T_{s(x,t)}}\leq t-\Einf{T_{s(x,t)}}\leq \Einf{T_{v(t)-1}}-\Einf{T_{s(x,t)}},
\end{equation}
and  we handle the two sides similarly. For the left hand side, we have
\begin{equation}\label{eq_1729}
\Einf{T_n}-\Einf{T_{[n+x\sqrt n]}}=\sum_{k=n}^{[ n+x\sqrt n]-1}\Esp_{k+1}(T_k)\equinf{n}\,\frac{x\sqrt{n}}{\mu_{n+1}},
\end{equation}
using forthcoming Lemma \ref{tek}  with $u_n= [ n+x\sqrt n]$. 
Moreover, applying forthcoming Lemma \ref{lemme_equivalents} with $f(y)=y$, $g(y)=[y+u\sqrt y]$ and $h(n)=\Vinf{T_n}$ and using \eqref{varTn}, we get
$$\Vinf{T_{[ n+x\sqrt n]}} \equinf{n} \Vinf{T_n}\equinf{n}  \frac{n}{(2\rho -1) \mu_{n+1}^2}.$$
Combining this equivalence with \eqref{eq_1729} and \eqref{encadrement_1515} yields 
\eqref{eq:1747} for $x>0$, while the case $x<0$ can be handled similarly. It ends up the proof.
\end{proof}


\begin{proof}[Proof of Corollary \ref{nrho}] We easily remark that  
 $$\Esp_{n+1}(T_n) \equinf{n} \frac{1}{n^{\rho}\log^{\gamma} n}\quad  ;
 \quad \Einf{T_n}\equinf{n}\sum_{k\geq n+1}\frac{1}{k^\rho\log^\gamma k}\equinf{n}
\frac{1}{(\rho-1)n^{\rho-1}\log^\gamma n}$$
and 
$$\Vinf{T_n}\equinf{n} \frac{1}{\big((2\rho-1)n^{2\rho-1}\log^{2\gamma} n\big)}.$$ 

\me The assumptions of Theorem \ref{corVR} can then be easily checked to get the result.
\end{proof}

\appendix
\section{Appendix : regularly varying functions} 
 The proofs  of the previous section rely on the  following technical results on regularly varying functions.

\me
\begin{lem}\label{VR_restes_sommes}
 Let $g$ be a function that varies regularly at $+\infty$ with index $\rho'<-1$. Then the series $\sum_{k\geq 0} g(k)$ converges and $R(n) = \sum_{k\geq n}g(k)$ varies regularly with index $\rho'+1$ and
$$\sum_{k\geq n}g(k) \equinf{n} - \frac{ng(n)}{\rho'+1}.$$
\end{lem}
\begin{proof}
  First, since $\rho'<-1$,  $\sum_{k\geq0}g(k)$ and $\int_0^{+\infty}g(x)\dif x$ are both convergent. Moreover,
  thanks to \cite[Thm 1.5.3]{Bingham1989}, a regularly varying function with negative index is equivalent to a non-increasing function. Then, without loss of generality, one can suppose that $g$ is non-increasing. Then, if $I_n:=\int_n^{+\infty}g(x)\dif x$, a classical comparison between series and integrals entails that
$1\leq \frac{ R_n}{I_n}\leq1 + \frac{g(n)}{I_n}.$
Using that $g$ varies regularly and according to \cite[Thm 1.5.11]{Bingham1989},
 \begin{equation}\label{eq:1536}
  \lim_{n\to+\infty}\frac{ng(n)}{I_n}= - (  \rho' +1).
  \end{equation}
Hence,  $I_n\sim R_n$ as $n\to+\infty$.
We also see from \eqref{eq:1536} that $I$ varies regularly at $+\infty$ with index $\rho'+1$. Since $I$ and $R$ are equivalent, $R$ also varies regularly with the same index.
\end{proof}

\me
\begin{lem}\label{lemme_equivalents}
Let $x_0\in[0,+\infty]$ and let $f$ and $g$ be two positive functions such that
$$f(x) \underset{x\to x_0}\longrightarrow L\in\{0,+\infty\}, \qquad \frac{f(x)}{g(x)} \underset{x\to x_0}\longrightarrow 1.$$
If $h$ varies regularly  at $L$, then
$$\frac{h(f(x))}{h(g(x))}\underset{x\to x_0}\longrightarrow 1.$$
Moreover, if $f(x)=f(x,t)=g(x)(1+t\eps(x))$ with $\lim_{x\to x_0}\eps(x)=0$, the previous convergence holds uniformly in $t$ in any compact subset of $\R$.
\end{lem}

\begin{proof}
We only prove the case $L=0$ and fix $\eps>0$. Thanks to Theorem 1.5.1 p.22 in \cite{Bingham1989}, the convergence given above in the definition 
of regularly varying function can be taken uniform with respect to $a$ in some compact set.
Then, there exist $\eta,\eta' >0$ such that for every $a \in [1-\eta, 1+\eta]$ and $y \in (0,\eta')$,
$$1-\eps \leq \frac{h(ay)}{h(y)} \leq 1+\eps.$$
Furthermore, for $x$ close enough to $x_0$, we have $g(x)\leq \eta'$ and $(1-\eta)\leq f(x)/g(x) \leq (1+\eta)$,
so that $$\left|\frac{h\left(g(x)\cdot \frac{f(x)}{g(x)}\right)}{h(g(x))}-1\right|\leq\eps,$$ which ends up the first part of the proof.
The second part follows in the same way since $1+t\eps(x)$ goes to $1$ uniformly in $t$ in any compact set.
\end{proof}

\begin{lem}\label{tek}%
Let $(m_n)_n$ be a regularly varying sequence and $(u_n)_n$ a sequence of integers such that $u_n\rightarrow \infty$  and $u_n/n \rightarrow 0$ as $n\rightarrow \infty$. Then
$$\sum_{k=n}^{n+u_n-1} m_k \equinf{n} u_nm_n.$$
\end{lem}
\begin{proof}We write
$$
\left|\sum_{k=n}^{n+u_n-1} m_k - u_nm_n\right| \leq m_n\sum_{k=n}^{n+u_n-1}\left|\frac{m_k}{m_n}-1\right|\\
\leq u_nm_n\sup_{t\in[0,1]}\left|\frac{m_{\lfloor n(1+tu_n/n)\rfloor}}{m_n}-1\right|.
$$
Using  the second part of  Lemma \ref{lemme_equivalents} with the regularly varying sequence
$(m_n)_n$, the last term vanishes, which ends up the proof.  
\end{proof}

\bi
\textbf{Acknowledgement.} 
This work  was partially funded by the Chaire Mod\'elisation Math\'ematique et Biodiversit\'e VEOLIA-\'Ecole Polytechnique-MNHN-F.X., by the labex LMH through the grant no ANR-11-LABX-0056-LMH in the "Programme des Investissements d'Avenir", by the professorial chair Jean Marjoulet and by  the project MANEGE `Mod\`eles
Al\'eatoires en \'Ecologie, G\'en\'etique et \'Evolution'
ANR-09-BLAN-0215.


\end{document}